\documentclass{amsart}
\usepackage{amscd}
\usepackage{amssymb}
\usepackage[all]{xy}
\usepackage{mathrsfs}
\usepackage{upgreek}
\usepackage{stmaryrd}
\usepackage[colorlinks = true]{hyperref}
\usepackage{rotating}
\usepackage{tikz}
\usepackage{tikz-cd}
\usepackage{bbm}
\usetikzlibrary{matrix,arrows,cd,decorations.pathmorphing}
\usepackage{mathrsfs}
\usepackage{centernot} %used for negating symbols (i.e. \centernot{\lhd} negates \lhd)
\usepackage{mathabx}

\newcommand{\bk}{\Bbbk}

\newcommand{\Z}{\mathbb{Z}}
\newcommand{\C}{\mathbb{C}}

\newcommand{\Gm}{\mathbb{G}_{\mathrm{m}}}

\newcommand{\sC}{\mathcal{C}}
\newcommand{\uC}{\underline{C}}
\newcommand{\uH}{\underline{H}}
\newcommand{\st}{\mathsf{st}}
\newcommand{\Parity}{\mathrm{Parity}}

\newcommand{\fg}{\mathfrak{g}}

\newcommand{\red}{{\mathrm{red}}}
\newcommand{\unip}{{\mathrm{unip}}}

\newcommand{\weyl}{\mathsf{M}}
\newcommand{\coweyl}{\mathsf{N}}
\newcommand{\tilt}{\mathsf{T}}

\newcommand{\cN}{\mathcal{N}}

\newcommand{\tcN}{{\widetilde{\mathcal{N}}}}

\newcommand{\bX}{\mathbf{X}}

\newcommand{\ext}{{\mathrm{ext}}}

\newcommand{\LV}{{\mathrm{LV}}}

\newcommand{\Wext}{W_\ext}

\newcommand{\cE}{\mathcal{E}}
\newcommand{\cF}{\mathcal{F}}
\newcommand{\cG}{\mathcal{G}}
\newcommand{\cL}{\mathcal{L}}
\newcommand{\cO}{\mathcal{O}}
\newcommand{\cV}{\mathcal{V}}
\newcommand{\cA}{\mathcal{A}}
\newcommand{\cS}{\mathcal{S}}
\newcommand{\cI}{\mathcal{I}}
\newcommand{\cH}{\mathcal{H}}
\newcommand{\cK}{\mathcal{K}}

\newcommand{\onabla}{\overline{\nabla}}
\newcommand{\oDelta}{\overline{\Delta}}

\newcommand{\IC}{\mathcal{IC}}

\newcommand{\cT}{\mathcal{T}}

\newcommand{\Fl}{\mathrm{Fl}}

\newcommand{\op}{{\mathrm{op}}}

\newcommand{\Rep}{\mathsf{Rep}}
\newcommand{\Irr}{\mathsf{Irr}}

\newcommand{\Coh}{\mathsf{Coh}}
\newcommand{\CohGGm}{\Coh^{G \times \Gm}}
\newcommand{\QCoh}{\mathsf{QCoh}}

\newcommand{\PCoh}{\mathsf{PCoh}}

\newcommand{\Tilt}{\mathsf{Tilt}}
\newcommand{\Silt}{\mathsf{Silt}}

\newcommand{\fD}{\mathfrak{D}}
\newcommand{\fS}{\mathfrak{S}}

\newcommand{\Db}{D^{\mathrm{b}}}

\newcommand{\D}{\mathbb{D}}
\newcommand{\cRHom}{R\mathcal{H}om}

\DeclareMathOperator{\Hom}{Hom}
\DeclareMathOperator{\Ext}{Ext}
\DeclareMathOperator{\Tor}{Tor}
\DeclareMathOperator{\End}{End}
\DeclareMathOperator{\Sym}{Sym}

\DeclareMathOperator{\supp}{supp}

\DeclareMathOperator{\Lie}{Lie}
\DeclareMathOperator{\codim}{codim}
\DeclareMathOperator{\rank}{rank}
\DeclareMathOperator{\res}{res}
\DeclareMathOperator{\ad}{ad}
\DeclareMathOperator{\cHom}{\mathcal{H}\mathit{om}}
\DeclareMathOperator{\cExt}{\mathcal{E}\mathit{xt}}

\newcommand{\id}{\mathrm{id}}
\newcommand{\simto}{\xrightarrow{\sim}}
\newcommand{\la}{\langle}
\newcommand{\ra}{\rangle}
\newcommand{\lb}{\mathord{\lbag}}
\newcommand{\rb}{\mathord{\rbag}}

\newcommand{\mbf}[1]{\mathbf{#1}}

\makeatletter
\def\lotimes{\@ifnextchar_{\@lotimessub}{\@lotimesnosub}}
\def\@lotimessub_#1{\mathchoice{\mathbin{\mathop{\otimes}^L}_{#1}}%
  {\otimes^L_{#1}}{\otimes^L_{#1}}{\otimes^L_{#1}}}
\def\@lotimesnosub{\mathbin{\mathop{\otimes}^L}}
\makeatother

\numberwithin{equation}{section}
\newtheorem{thm}{Theorem}[section]
\newtheorem{lem}[thm]{Lemma}
\newtheorem{prop}[thm]{Proposition}
\newtheorem{cor}[thm]{Corollary}

\theoremstyle{definition}
\newtheorem{defn}[thm]{Definition}

\theoremstyle{remark}
\newtheorem{rmk}[thm]{Remark}

\title[Silting complexes and the Humphreys conjecture]{Silting complexes of coherent sheaves and the Humphreys conjecture}

 \author{Pramod N. Achar}
 \address{Department of Mathematics\\
   Louisiana State University\\
   Baton Rouge, LA 70803\\
   U.S.A.}
 \email{pramod@math.lsu.edu}

 \author{William Hardesty}
 \address{School of Mathematics and Statistics\\
   University of Sydney\\
   Camperdown, NSW 2006\\
   U.S.A.}
 \email{hardes11@gmail.com}
 
 \thanks{P.A. was supported by NSF Grant No.~DMS-1802241. W.H. was supported by the ARC Discovery Grant No.~DP170104318.}

\begin{document}

\begin{abstract}
Let $G$ be a connected reductive algebraic group over an algebraically closed field $\bk$ of characteristic $p \ge 0$, and let $\cN$ be its nilpotent cone.  Under mild hypotheses, we construct for each nilpotent $G$-orbit $C$ and each indecomposable tilting vector bundle $\cT$ on $C$ a certain complex $\cS(C,\cT) \in \Db\CohGGm(\cN)$.  We prove that these objects are (up to shift) precisely the indecomposable objects in the coheart of a certain co-$t$-structure.  

We then show that if $p$ is larger than the Coxeter number, then the hypercohomology $H^\bullet(\cN, \cS(C,\cT))$ is identified with the cohomology of a tilting module for $G$.  This confirms a conjecture of Humphreys on the support of the cohomology of tilting modules.
\end{abstract}

\maketitle

%%%%%%%%%%%%%%%%%%%%%%%%%%%%%%%%%%%%%%%%%%%%%%%%%%%%%%%%%%%%%%%%%%%%%%%%%%%
\section{Introduction}
%%%%%%%%%%%%%%%%%%%%%%%%%%%%%%%%%%%%%%%%%%%%%%%%%%%%%%%%%%%%%%%%%%%%%%%%%%%

%--------------------------------------------------------------------------
\subsection{The Humphreys conjecture}
%--------------------------------------------------------------------------

Let $\mbf{G}$ be a connected reductive group over an algebraically closed field $\bk$ of characteristic $p$.  Assume that $p$ is larger than the Coxeter number $h$ for $\mbf{G}$.  Let $\mbf{G}_1$ be its first Frobenius kernel, and let $G = \mbf{G}/\mbf{G}_1$ be its Frobenius twist.  Let $\cN$ be the nilpotent variety in the Lie algebra of $G$.  It is well known that the algebra $\Ext^\bullet_{\mbf{G}_1}(\bk,\bk)$ is ($G$-equivariantly) isomorphic to the coordinate ring $\bk[\cN]$.  As a consequence, for any $\mbf{G}$-module $M$, the $\mbf{G}_1$-cohomology
\[
H^\bullet(\mbf{G}_1,M) = \Ext^\bullet_{\mbf{G}_1}(\bk,M)
\]
has the structure of a $G$-equivariant graded $\bk[\cN]$-module, or equivalently, a $G \times \Gm$-equivariant coherent sheaf on $\cN$.  

The main goal of this paper is to give a new description of this cohomology in the case where $M$ is an indecomposable tilting $\mbf{G}$-module.  This cohomology vanishes except for tilting modules of the form $M = \tilt(w_\lambda \cdot 0)$ (see Section~\ref{sec:humphreys} for this notation), where $\lambda$ is a dominant weight.  Using the results of~\cite{ar:rglg}, one can refine this problem as follows: for each $\lambda$, there is an object
\[
\fS_\lambda \in \Db\CohGGm(\cN)
\]
equipped with a canonical isomorphism (see~\cite[Propositions~9.4 and 9.5]{ar:dkf})
\[
R^\bullet\Gamma(\cN, \fS_\lambda) \cong H^\bullet(\mbf{G}_1, \tilt(w_\lambda \cdot 0)).
\]

The \emph{relative Humphreys conjecture} is a conjectural description of the support of $H^\bullet(\mbf{G}_1, \tilt(w_\lambda \dot 0))$ (or, equivalently, the support of $\fS_\lambda$), in terms of the combinatorics of Kazhdan--Lusztig cells.  Here is a brief summary of previous results on this conjecture:
\begin{description}
\item[Quantum case] In~\cite{bez:ctm}, Bezrukavnikov gave a description of the complex version $\fS_\lambda^\C$, and thereby proved the quantum-group analogue of the relative Humphreys conjecture.

\item[Reductive groups for $p \gg 0$] In~\cite{ahr:hcsv}, the authors and S.~Riche proved that the relative Humphreys conjecture is true when $p$ is ``large enough,'' i.e., larger than some unknown bound depending on $\mbf{G}$.  The proof involves a reduction to the quantum case studied by Bezrukavnikov.

\item[$\mathbf{GL}_n$ for $p > h = n+1$] The relative Humphreys conjecture for $\mathrm{GL}_n$ follows from work of the second author~\cite{hardesty}, as explained in~\cite[Remark~9.4(1)]{ahr:hcsv}.  A second and rather different proof was obtained by the authors in~\cite{ah:cot1}.
\end{description}

In this paper, we prove the relative Humphreys conjecture in full generality, for all reductive groups and all $p > h$.  The proof, which is based on a new description of the $\fS_\lambda$ in terms of ``silting complexes,'' is independent of the main arguments in~\cite{ahr:hcsv, bez:ctm}.  A side effect of the proof is an explicit description  of the coherent sheaf $H^\bullet(\mbf{G}_1,\tilt(w_\lambda \dot 0))$ over the open orbit in its support, conjectured in~\cite{ahr:ctmap}.

We remark on two other problems that are not addressed in this paper:
\begin{enumerate}
\item Humphreys originally proposed a conjectural description of $\Ext^\bullet_{\mbf{G}_1}(M,M)$ rather than of $\Ext^\bullet_{\mbf{G}_1}(\bk,M)$: see~\cite{hum:cmr}.  The original Humphreys conjecture has been proved for $\mathrm{GL}_n$ in~\cite{hardesty}, and for any $\mbf{G}$ when $p \gg 0$ in~\cite{ahr:hcsv}.  See~\cite[Lemma~8.11 and Remark~9.4]{ahr:hcsv} for the relationship between the original and relative Humphreys conjectures.

\item In~\cite{ah:cot1}, the authors proposed a \emph{scheme-theoretic Humphreys conjecture}, asserting that the scheme-theoretic support of $\fS_\lambda$ is reduced, and they proved this conjecture for $\mbf{G} = \mathrm{GL}_n$. 
\end{enumerate}
For general $\mbf{G}$ and $p > h$, the original and scheme-theoretic Humphreys conjectures  both remain open.

%--------------------------------------------------------------------------
\subsection{Silting complexes on the nilpotent cone}
%--------------------------------------------------------------------------

The main geometric arguments in the paper involve only $G$ and $\cN$ (and not $\mbf{G}$), and are valid under much milder assumptions the characteristic $p$ of $\bk$.  From now on, we assume only that $p$ is ``pretty good'' for $G$ (see Section~\ref{sec:prelim}).

A \emph{silting subcategory} of a triangulated category is an additive subcategory whose objects enjoy certain strong $\Ext$-vanishing conditions.  In~\cite{ah:cot1}, the authors showed that the category of direct sums of objects of the form $\fS_\lambda[n]\la -n\ra$ is a silting subcategory of $\Db\CohGGm(\cN)$, called the \emph{supportive silting subcategory}.  

In this paper, we construct a new silting subcategory of $\Db\CohGGm(\cN)$, which we call the \emph{orbitwise silting subcategory}, because it has a geometric description that proceeds ``one nilpotent orbit at a time.''   This description involves the notion of a tilting vector bundle on a nilpotent orbit, defined in Section~\ref{ss:cot-nilp} below.  We show that for any nilpotent orbit $C \subset \cN$ and any indecomposable tilting vector bundle $\cT \in \Tilt(\CohGGm(C))$, there is a unique way to extend $\cT$ to an indecomposable object
\[
\cS(C,\cT) \in \Db\CohGGm(\cN)
\]
that is supported on $\overline{C}$ and satisfies appropriate $\Ext$-vanishing conditions with tilting vector bundles on smaller orbits.  Objects of the form $\cS(C,\cT)[n]\la -n\ra$ are precisely the indecomposable objects in the orbitwise silting subcategory.

The main geometric theorem of the paper states that the supportive and orbitwise silting subcategories actually coincide.  Thus, for each pair $(C,\cT)$ as above, there is a unique dominant weight $\lambda$ such that
\[
\cS(C,\cT) \cong \fS_\lambda.
\]
We also prove that this correspondence is given by the \emph{Lusztig--Vogan bijection}.

If we now assume that $p > h$, then we have
\[
H^\bullet(\mbf{G}_1, \tilt(w_\lambda \cdot 0)) \cong R^\bullet\Gamma(\cN, \cS(C,\cT)).
\]
The relative Humphreys conjecture is essentially a corollary of this formula.

%--------------------------------------------------------------------------
\subsection{Application to the \texorpdfstring{$p$}{p}-canonical basis}
%--------------------------------------------------------------------------

The silting complexes introduced in this paper can be thought of as a coherent counterpart to the $p$-canonical basis of the affine Hecke algebra~\cite{jw} and the theory of parity sheaves~\cite{jmw:ps}, both of which play prominent roles in recent developments in modular representation theory.  These parallels are summarized in the following table.

\medskip
\begin{center}
\small
\begin{tabular}{l|c|c|c}
& \it bases for the &  \it constructible sheaves & \it coherent sheaves on \\
& \it affine Hecke algebra & \it on flag varieties & \it the nilpotent cone \\
\hline
char.~$0$ & Kazhdan--Lusztig basis & perverse sheaves & perverse-coherent sheaves \\
\hline
char.~$p$ & $p$-canonical basis & parity sheaves & silting complexes
\end{tabular}
\end{center}
\medskip

In fact, one can make a more precise statement: it turns out that the Grothen\-dieck group $K(\Db\CohGGm(\cN))$ is naturally a quotient of the affine Hecke algebra, and this quotient map sends the Kazhdan--Lusztig basis to the basis of simple perverse-coherent sheaves.  We will see at the end of this paper that this map also sends the $p$-canonical basis to the basis of silting complexes.  This observation implies a certain positivity property for the $p$-canonical basis, and it suggests conjectural avenues for further study around the themes of ($p$-)Kazhdan--Lusztig cells, truncated convolution of perverse sheaves, and vector bundles on nilpotent orbits.

%--------------------------------------------------------------------------
\subsection{Notation and terminology}
%--------------------------------------------------------------------------

If $V = \bigoplus_{j \in \Z} V_j$ is a graded vector space, we define $V\la n\ra$ to be the graded vector space given by $(V\la n\ra)_j = V_{n+j}$.  Equivalently, if we think of $V$ as a $\Gm$-representation, then $V\la n\ra = V \otimes \bk_{-n}$, where $\bk_{-n}$ is the $1$-dimensional representation where $\Gm$ acts with weight $-n$.  Similar notation is used for coherent sheaves.

In this paper, as in its predecessor~\cite{ah:cot1}, we use the term \emph{silting object} to mean any object in a silting subcategory, not just a generator.  See~\cite[Remark~2.2]{ah:cot1} for context.  The notion of a silting subcategory is equivalent to that of a bounded co-$t$-structure, and we mainly use the latter notion in this paper.  See~\cite[\S2]{ah:cot1} for additional background and references on co-$t$-structures.

%--------------------------------------------------------------------------
\subsection{Contents of the paper}
%--------------------------------------------------------------------------

We begin in Section~\ref{sec:gen-co-t} with a ``toy example'' (needed later in the paper) of a co-$t$-structure on representations of certain nonreductive groups.  Section~\ref{sec:prelim} contains preliminaries and notation related to the nilpotent cone, and Sections~\ref{sec:lie-extend}, \ref{sec:nil-extend}, and~\ref{sec:nilcone} are devoted to constructing co-$t$-structures on coherent sheaves in increasingly difficult settings, culminating with the orbitwise co-$t$-structure on $\Db\CohGGm(\cN)$, obtained in Theorem~\ref{thm:orbitwise}.  In Section~\ref{sec:lv}, we prove that this co-$t$-structure coincides with the supportive co-$t$-structure from~\cite{ah:cot1}, and we describe the combinatorics of the relationship between the two.  Section~\ref{sec:humphreys} contains the proof of the relative Humphreys conjecture, and Section~\ref{sec:pcanonical} discusses potential applications to the study of the $p$-canonical basis.  Finally, Appendix~\ref{sec:duality} contains a technical lemma on co-$t$-structures that is needed in Section~\ref{sec:lv}.

%%%%%%%%%%%%%%%%%%%%%%%%%%%%%%%%%%%%%%%%%%%%%%%%%%%%%%%%%%%%%%%%%%%%%%%%%%%
\section{Group representations}
\label{sec:gen-co-t}
%%%%%%%%%%%%%%%%%%%%%%%%%%%%%%%%%%%%%%%%%%%%%%%%%%%%%%%%%%%%%%%%%%%%%%%%%%%

For an algebraic group $H$ over an algebraically closed field $\bk$, let $\Rep(H)$ be the category of finite-dimensional algebraic representations.  If $H$ is connected and reductive, then (as observed in, say,~\cite[Remark~2.6]{ah:cot1}), the category $\Tilt(H) \subset \Rep(H)$ is the coheart of a co-$t$-structure.  If $H$ is disconnected, and if the characteristic of $\bk$ does not divide $|H/H^\circ|$, then $\Rep(H)$ is again a highest-weight category (see~\cite[Theorem~3.7]{ahr:rdg}), and $\Tilt(H)$ is again the coheart of a co-$t$-structure.  The goal of this section is to extend these observations to certain nonreductive groups.  

From now on, let $H$ be a (possibly disconnected) algebraic group over $\bk$, which is equipped with a Levi decomposition $H = H_\red \ltimes H_\unip$, where $H_\unip$ is a connected unipotent group, and $H_\red$ is a possibly disconnected group whose identity component $H_\red^\circ$ is reductive.  Suppose we are given an action of $\Gm$ on $H$ by group automorphisms, so that it makes sense to form the group $\Gm \ltimes H$.  We impose the following assumptions:
\begin{enumerate}
\item The characteristic of $\bk$ does not divide $|H_\red/H_\red^\circ|$.
\item $\Gm$ acts trivially on $H_\red$.
\item The induced action of $\Gm$ on $\Lie(H_\unip)$ has strictly positive
weights.
\end{enumerate}
We will construct a co-$t$-structure on $\Rep(\Gm \ltimes H)$.  Thanks to the first assumption, the highest-weight theory of~\cite[Theorem~3.7]{ahr:rdg} is available for $H_\red$.  Let $\Irr(H_\red)$ be the set of isomorphism classes of irreducible $H_\red$-modules.  (An explicit parametrization of this set in terms of weights for $H_\red^\circ$ is given in~\cite[Theorem~2.16]{ahr:rdg}.)  For each $\omega \in \Irr(H_\red)$, let $\weyl_\omega$, $\coweyl_\omega$, and $\tilt_\omega$ denote the corresponding standard, costandard, and indecomposable tilting module, respectively.

\begin{lem}\label{lem:tilting-silt}
Let $\omega, \upsilon \in \Irr(H_\red)$, and regard $\weyl_\omega$ and $\coweyl_\upsilon$ as $\Gm \ltimes H$-modules with trivial $\Gm$-action.  In $\Db\Rep(\Gm \ltimes H)$, we have $\Hom(\weyl_\omega, \coweyl_\upsilon[n]\la k\ra) = 0$ whenever $n + k > 0$.
\end{lem}
\begin{proof}
If $H$ is connected and $H_\unip$ is trivial (i.e., if $\Gm \ltimes H$ is a connected reductive group), then this is a standard result in the representation theory of reductive groups: see, for instance,~\cite[\S\S II.4.9--II.4.13]{jan:rag}.  That proof can be adapted to the case where $H_\unip$ is nontrivial as well.  We briefly indicate the main steps below.

Assume for now that $H$ is connected.  Choose a maximal torus and a Borel subgroup $T \subset B \subset H_\red$.  Then $\Gm \times T$ is a maximal torus of $\Gm \ltimes H$.  Let $\Phi$ be the root system of $H_\red$, and let $\Phi^+$ be the set of positive roots corresponding to the \emph{opposite} of $B$.  Let $\tilde B = B \ltimes H_\unip$; this is a Borel subgroup of $H$.  

Choose a cocharacter $\rho: \Gm \to T$ such that $\la \alpha, \rho \ra > 0$ for all $\alpha \in \Phi^+$.  Next, choose a positive integer $m$ such that for every $T$-weight $\beta$ on $\Lie(H_\unip)$, we have $\la \beta,\rho\ra < m$, and then let $\tilde\rho: \Gm \to \Gm \times T$ be the map $\tilde\rho(z) = (z^{-m}, \rho(z))$.  Then the pairing of $\tilde\rho$ with every $\Gm \times T$-weight on $\Lie(\tilde B)$ is strictly negative.

By a minor variant on the proof of~\cite[Lemma~II.4.9]{jan:rag}, one can show that the trivial $\tilde B$-module $\bk$ admits an injective resolution
\[
0 \to \bk \to I^0 \to I^1 \to \cdots
\]
such that for any $\Gm \times T$-weight $\gamma$ occurring in $I^n$, we have $\la \tilde\rho,\gamma \ra \ge n$.  Now let $\upsilon$ be a dominant weight for $H_\red$.  Using the injective resolution above, the proof of~\cite[Proposition~II.4.10(b)]{jan:rag} shows that
\[
\Ext^n_{\tilde B}(\bk, \bk_\upsilon\la k\ra) \ne 0
\qquad\text{implies}\qquad
-mk -\la \rho,\upsilon\ra \ge n,
\]
or equivalently, $n+k \le -(m-1)k - \la \rho,\upsilon\ra$.  Since $\upsilon$ is dominant, this in turn implies that $n+k \le 0$.  In other words, if $n + k > 0$, then the group
\[
\Ext^n_{\tilde B}(\bk, \bk_\upsilon\la k\ra)
\cong \Ext^n_H(\bk, \mathrm{ind}_{\tilde B}^H \bk_\upsilon \la k\ra) = \Hom(\bk, \coweyl_\upsilon[n]\la k\ra
\]
vanishes.  This immediately implies, more generally, that if $N$ is any $H_\red$-rep\-re\-sen\-ta\-tion with a good filtration, then $\Hom(\bk,\coweyl[n]\la k\ra) = 0$ whenever $n + k > 0$.  In particular, we have
\[
\Hom(\weyl_\omega, \coweyl_\upsilon[n]\la k\ra) \cong \Hom(\bk, \weyl_\omega^* \otimes \coweyl_\upsilon[n]\la k\ra) = 0.
\]
We have completed the proof in the case where $H$ is connected.

If $H$ is disconnected, then, thanks to our assumption that the characteristic of $\bk$ does not divide the order of $H/H^\circ$, we have a natural isomorphism
\begin{equation}\label{eqn:disconn-ext}
\Hom_H(\weyl_\omega, \coweyl_\upsilon[n]\la k\ra) \cong \big(\Hom_{H^\circ}(\weyl_\omega, \coweyl_\upsilon[n]\la k\ra)\big)^{H/H^\circ}.
\end{equation}
(See, for instance,~\cite[Lemma~2.18]{ahr:rdg} for more explanation.)  As an $H^\circ_\red$-module, $\weyl_\omega$ (resp.~$\coweyl_\upsilon$) is a direct sum of Weyl (resp.~dual Weyl) modules, so the right-hand side vanishes by the case of connected groups considered above.
\end{proof}

\begin{prop}\label{prop:tilting-silt}
There is a unique co-$t$-structure on $\Db\Rep(\Gm \ltimes H)$ whose indecomposable silting objects are precisely the objects of the form $\tilt_\omega[n]\la -n\ra$ for $\omega \in \Irr(H_\red)$ and $n \in \Z$.
\end{prop}
\begin{proof}
Let $\fS \subset \Db\Rep(\Gm \ltimes H)$ be the full additive subcategory consisting of direct sums of objects of the form $\tilt_\omega[n]\la -n\ra$.  It is immediate from Lemma~\ref{lem:tilting-silt} that for $S, S' \in \fS$, we have $\Hom(S,S'[k]) = 0$ for all $k > 0$.  Since $\fS$ generates $\Db\Rep(\Gm \ltimes H)$ as a triangulated category, it is a silting subcategory, and hence the coheart of a unique co-$t$-structure: cf.~\cite[Proposition~2.5]{ah:cot1}.
\end{proof}

We denote the coheart of the co-$t$-structure from Proposition~\ref{prop:tilting-silt} by
\[
\Silt(\Gm \ltimes H) := \Db\Rep(\Gm \ltimes H)_{\ge 0} \cap \Db\Rep(\Gm \ltimes H)_{\le 0}.
\]

\begin{rmk}\label{rmk:weyl-cot}
With a little bit of extra work, one can show that the co-$t$-structure from Proposition~\ref{prop:tilting-silt} has the following description:
\[
\begin{aligned}
\Db\Rep(\Gm \ltimes H)_{\ge 0} &=
\begin{array}{@{}c@{}}
\text{the full subcategory of $\Db\Rep(\Gm\ltimes H)$ generated} \\
\text{under extensions by $\weyl_\omega[n]\la k \ra$ with $n + k \le 0$}
\end{array} \\
\Db\Rep(\Gm \ltimes H)_{\le 0} &=
\begin{array}{@{}c@{}}
\text{the full subcategory of $\Db\Rep(\Gm\ltimes H)$ generated} \\
\text{under extensions by $\coweyl_\omega[n]\la k \ra$ with $n + k \ge 0$}
\end{array}
\end{aligned}
\]
\end{rmk}

We will need a lemma about the co-$t$-structure defined above in terms of the following notion.

\begin{defn}
For a module $N$ of a (possibly disconnected) reductive group $H_\red$, the \emph{good filtration dimension} of $N$ is defined to be the 
smallest integer $k$ such that for all $j \geq k+1$ and any Weyl module $M$,
$
\Ext_{H_\red}^j(M, N) = 0.
$
\end{defn}

For other characterizations of good filtration dimension, see~\cite[Proposition~3.4]{fp:claag}.  (That paper assumes that $H_\red$ is connected, and imposes some restrictions on the characteristic of $\bk$, depending on the type of $H_\red$, because it was not known at the time that tensor product preserves the property of having a good filtration in full generality.  In fact,~\cite[Proposition~3.4]{fp:claag} holds in general for possibly disconnected reductive groups $H_\red$, as long as the characteristic of $\bk$ does not divide $|H_\red/H_\red^\circ|$.)

\begin{lem}\label{lem:gfd-co-t}
Let $N$ be an $H_\red$-module with good filtration dimension $\leq i$, regarded as a $\Gm \ltimes H$-module with trivial $\Gm$-action. Then, for any $j \in \Z$, we have
\[
N\langle j \rangle \in \Db\Rep(\Gm \ltimes H)_{\leq i-j}
\]
\end{lem}
\begin{proof}
We proceed by induction on $i$.  If $i = 0$, i.e., if $N$ has a good filtration, then Lemma~\ref{lem:tilting-silt} implies that $\Hom(\tilt_\omega[n+j-1]\la -n\ra, N\la j\ra) = 0$ for any $j$.  In other words, $N\la j\ra \in \Db\Rep(\Gm \ltimes H)_{\leq -j}$.

Now suppose $i \geq 1$, and that the result holds for any module with the good filtration dimension $\leq i-1$. Given $N$ with 
good filtration dimension $\le i$, let $E$ be a module with a good filtration such that there is an embedding $N \hookrightarrow E$ 
(such modules always exist since the rationally injective $H_\red$-modules have good filtrations). This gives a short exact sequence 
\[
0 \to  N \rightarrow E \rightarrow K \rightarrow  0,
\]
where the cokernel $K$ has good filtration dimension $\leq i-1$. Rotating this triangle gives 
\[
K[-1]\la j\ra \rightarrow N\la j\ra \rightarrow E\la j\ra \rightarrow .
\]
By induction, the first term lies in $\Db\Rep(\Gm \ltimes H)_{\le i-j}$, and the last term lies in $\Db\Rep(\Gm \ltimes H)_{\le -j}$, so we must have
 that $M \in \Rep(\Gm \times H)_{\leq i}$ as well. 
\end{proof}

%%%%%%%%%%%%%%%%%%%%%%%%%%%%%%%%%%%%%%%%%%%%%%%%%%%%%%%%%%%%%%%%%%%%%%%%%%%
\section{Preliminaries on nilpotent orbits}
\label{sec:prelim}
%%%%%%%%%%%%%%%%%%%%%%%%%%%%%%%%%%%%%%%%%%%%%%%%%%%%%%%%%%%%%%%%%%%%%%%%%%%

Let $G$ be a connected reductive group over an algebraically closed field $\bk$, and let $\fg$ be its Lie algebra. We assume throughout the paper that
\[
\text{The characteristic $p$ of $\bk$ is pretty good for $G$.}
\]
For the definition of ``pretty good,'' see~\cite[Definition~2.11]{herpel}.  This condition is equivalent to requiring $G$ to be ``standard'' in the sense of~\cite[\S4]{mt2}.  According to~\cite[Lemma~2.3]{ahr:ies}, \cite[Lemma~2.12]{herpel}, and~\cite[Proposition~12]{mt1}, this assumption implies the following commonly used conditions on $G$:
\begin{enumerate}
\item There exists a separable isogeny $\tilde G \to G$, where the derived subgroup of $\tilde G$ is simply connected.
\item The characteristic of $\bk$ is good for $G$.
\item There exists a nondegenerate $G$-invariant bilinear form on $\fg$.
\end{enumerate}
In this section, we establish notation and review some relevant results about the nilpotent cone $\cN \subset \fg$.

%--------------------------------------------------------------------------
\subsection{Coherent sheaves on the Lie algebra and the nilpotent cone}
%--------------------------------------------------------------------------

If $X$ is any $G \times \Gm$-variety, we denote by $\CohGGm(X)$ the category of $G \times \Gm$-equivariant coherent sheaves on $X$.  If $Z \subset X$ is a $G \times \Gm$-stable closed subset, we denote by
\[
\Db_Z\CohGGm(X) \subset \Db\CohGGm(X)
\]
the full triangulated subcategory of $\Db\CohGGm(X)$ consisting of objects supported set-theoretically on $Z$.

We make $\fg$ into a $G \times \Gm$-variety by letting $\Gm$ act with weight $-2$.  Next, let $\cN \subset \fg$ be the nilpotent cone, and let $C \subset \fg$ be a nilpotent orbit.  Let $\partial C = \overline{C} \smallsetminus C$, and then let
\[
\fg_C = \fg \smallsetminus \partial C,
\qquad
\cN_C = \cN \smallsetminus \partial C.
\]
Thus, $\fg_C$ and $\cN_C$ are open subsets of $\fg$ and $\cN$, respectively, in which $C$ embeds as a closed subvariety.  Let
\[
j_C: C \hookrightarrow \fg_C,
\]
be the inclusion map.  All of these spaces are preserved by the  $\Gm$-action, so we can regard them as $G \times \Gm$-subvarieties of $\fg$.  

Choose a point $x_C \in C$ and an associated cocharacter $\phi_{x_C}: \Gm \to G$.  The stabilizer $G^{x_C}$ admits a Levi decomposition
\[
G^{x_C}= G^{x_C}_\red \ltimes G^{x_C}_\unip,
\]
where $G^{x_C}_\red$ is the centralizer of $\phi_{x_C}$ in $G^{x_C}$.  The group $G^{x_C}_\red$ may be disconnected, but the assumptions from the beginning of the section imply that the characteristic of $\bk$ does not divide $|G^{x_C}_\red/(G^{x_C}_\red)^\circ|$.

Recall that a for a closed (possibly disconnected) reductive subgroup $H$ of $G$, we call $(G,H)$ a \emph{Donkin pair} if for any $G$-module $V$ with a good filtration, 
the restriction $\res^G_H(V)$ has a good filtration for $H$ (see \cite[II.4.22]{jan:rag}).  This condition implies, more generally, that if $V$ has good filtration dimension${}\le i$ for $G$, then $\res^G_H(V)$ has good filtration dimension${}\le i$ for $H$.

The following result from~\cite[Corollary~1.2]{ah:goodfilt} will play a crucial role in Section~\ref{sec:lie-extend}.

\begin{thm}[{\cite{ah:goodfilt}}]\label{thm:donkin}
For $x_C$ as above, the pair $(G, G^{x_C}_\red)$ is a Donkin pair.
\end{thm}

%--------------------------------------------------------------------------
\subsection{The co-\texorpdfstring{$t$}{t}-structure on a nilpotent orbit}
\label{ss:cot-nilp}
%--------------------------------------------------------------------------

We have an isomorphism
\begin{equation}\label{eqn:semi-split}
\Gm \ltimes G^{x_C} \simto (G \times \Gm)^{x_C}
\qquad\text{given by}\qquad
t \ltimes g \mapsto (\phi_{x_C}(t)g, t),
\end{equation}
where the semidirect product $\Gm \ltimes G^{x_C}$ is defined by having $\Gm$ act by $z \cdot g = \phi_{x_C}(z)g\phi_{x_C}(z)^{-1}$.  This action is trivial on $G^{x_C}_\red$ and has strictly positive weights on $\Lie(G^{x_C}_\unip)$.  

The action of $G \times \Gm$ on $C$ induces an isomorphism $(G \times \Gm)/(G \times \Gm)^{x_C} \cong C$ (see~\cite[\S2.9]{jan:nort}).  It follows that there are equivalences of categories
\begin{equation}\label{eqn:orbit-equiv}
\CohGGm(C) \cong \Rep((G \times \Gm)^{x_C}) \cong \Rep(\Gm \ltimes G^{x_C}),
\end{equation}

Via these equivalences, we can transfer results from Section~\ref{sec:gen-co-t} to $\CohGGm(C)$.  In particular, Proposition~\ref{prop:tilting-silt} gives us a co-$t$-structure on $\Db\CohGGm(C)$ whose coheart is denoted by
\begin{equation}\label{eqn:siltc-defn}
\Silt(C) = \Db\CohGGm(C)_{\ge 0} \cap \Db\CohGGm(C)_{\le 0}.
\end{equation}

Let us introduce some notation to label the indecomposable objects in $\Silt(C)$.  An object $\cT \in \CohGGm(C)$ is called a \emph{tilting vector bundle} if it corresponds under~\eqref{eqn:orbit-equiv} to a tilting module for $\Gm \times G^{x_C}_\red$ (with trivial action of $G^{x_C}_\unip$).  
Let
\[
\Omega_C = \text{the set of isomorphism classes of irreducible $G^{x_C}_\red$-representations,}
\]
and for $\omega \in \Omega_C$, let $\cT_\omega \in \CohGGm(C)$ be the corresponding indecomposable tilting vector bundle.  Then the indecomposable objects in $\Silt(C)$ are precisely those of the form 
$\cT_\omega[n]\la -n\ra$ for $\omega \in \Omega_C$ and $n \in \Z$.

% We will be using the following assumption throughout this section.  
%\begin{assump}\label{assump:donkin-pair}
%For any orbit $C \subset \cN$ and $x_C \in C$, then $(G^{x_C}_\red, G)$ is a Donkin pair. 
%\end{assump}
%\begin{rmk}
%This holds for $p$ sufficiently large, since it suffices to show that every fundamental representation of $G$ restricts to a
%module with a good filtration for $G^{x_C}_\red$. (In fact, we can replace $G^{x_C}_\red$ with its connected component $H \subseteq G^{x_C}_\red$.) 
%In this case, when $p$ is large enough, all of these representations will be semisimple and tilting. 
%\end{rmk}

%--------------------------------------------------------------------------
\subsection{Serre--Grothendieck duality}
%--------------------------------------------------------------------------

The Serre--Grothendieck duality functor on $\cN$ is the functor
\[
\D = \D_\cN: \Db\CohGGm(\cN) \to \Db\CohGGm(\cN)
\quad\text{given by}
\quad
\D_\cN = \cRHom({-},\cO_\cN).
\]
This definition involves a choice; see~\cite[\S 3.2]{ah:pgl3} for a discussion.  Other variants of these functors that we will need include
\begin{align*}
\D_C = \cRHom({-},\cO_C[-\codim C]\la \codim C\ra) &: \Db\CohGGm(C) \to \Db\CohGGm(C), \\
\D_\fg = \cRHom({-},\cO_\fg[\rank G]\la -\rank G\ra) &: \Db\CohGGm(\fg) \to \Db\CohGGm(\fg).
\end{align*}
Here, $C$ is a nilpotent orbit, and $\codim C$ is defined to be its codimension in $\cN$.

For compatibilities among these functors, let $j: \cN \hookrightarrow \fg$ and $i_C: C \hookrightarrow \cN$ be the inclusion maps.  Then we have
\[
j^!\cO_\fg[\rank G]\la -\rank G\ra \cong \cO_\cN,
\qquad
i_C^!\cO_\cN \cong \cO_C[-\codim C]\la \codim C\ra.
\]
(For a proof of the latter, see~\cite[Corollary~2.5]{ah:pgl3}; very similar reasoning yields the former isomorphism as well.)  It follows that
\begin{equation}\label{eqn:serre-groth-compat}
\D \circ j_* \cong j_* \circ \D
\qquad\text{and}\qquad
j^* \circ \D \cong \D \circ j^!.
\end{equation}
One can also define $\D$ on open subsets of $\fg$ or $\cN$. For instance, on $\fg_C$, we have
\begin{equation}\label{eqn:serre-groth-open}
\D \circ j_{C*} \cong j_{C*} \circ \D
\qquad\text{and}\qquad
j_C^* \circ \D \cong \D \circ j_C^!.
\end{equation}
Now let $U \subset \cN$ let a $G$-stable open subset, and let $C$ be a nilpotent orbit that is closed in $U$.  If we let $i_C: C \hookrightarrow U$ be the inclusion map, then
\begin{equation}\label{eqn:serre-groth-orbit}
\D \circ i_{C*} \cong i_{C*} \circ \D.
\end{equation}
There are also statements involving $i_C^*$ and $i_C^!$, but because $U$ is usually not smooth, these functors require working in $D^-\CohGGm(U)$ or $D^+\CohGGm(U)$, rather than $\Db\CohGGm(U)$.  We will mostly avoid unbounded derived categories in this paper.

%--------------------------------------------------------------------------
\subsection{Opposition}
%--------------------------------------------------------------------------

Recall that an \emph{opposition} is an involutive automorphism $\sigma: G \to G$ that preserves some maximal torus $T \subset G$ and satisfies $\sigma(t) = t^{-1}$ for all $t \in T$.  The existence of an opposition follows from~\cite[II.1.16]{jan:rag}.  Given a $G$-module $V$, let $V^\sigma$ denote the representation obtained by twisting the $G$-action by $\sigma$.  See~\cite[\S 4]{ah:pgl3} for a discussion of how to extend this construction to a functor
\[
({-})^\sigma: \Db\CohGGm(\cN) \to \Db\CohGGm(\cN).
\]
According to~\cite[Corollary~4.2]{ah:pgl3}, this functor preserves supports.    
Moreover, its action on coherent sheaves on an orbit can be described explicitly described using~\cite[Lemma~4.3]{ah:pgl3}, which yields for each nilpotent orbit $C \subset \cN$ an involutive automorphism
\[
\id \times \sigma_C: \Gm \ltimes G^{x_C} \to \Gm \ltimes G^{x_C}.
\]
Let $\cF \in \CohGGm(\cN)$ be such that $\cF|_{\cN_C}$ is supported scheme-theoretically on $C$.  Thus, we can regard $\cF|_C$ as an object of $\Rep(\Gm \ltimes G^{x_C})$ via~\eqref{eqn:orbit-equiv}.  Then~\cite[Lemma~4.3]{ah:pgl3} implies that
\begin{equation}\label{eqn:sigma-orbit}
(\cF^\sigma)|_C \cong (\cF|_C)^{\id \times \sigma_C}.
\end{equation}

\begin{lem}\label{lem:tilting-oppo}
Let $T$ be a tilting $G^{x_C}_\red$-module, regarded as a $\Gm \ltimes G^{x_C}$-module with trivial $\Gm$-action.  Then $T^{\id \times \sigma_C} \cong T^*$.
\end{lem}
\begin{proof}
The analogous statement for irreducible representations of $\Gm \ltimes G^{x_C}$ is shown in the proof of~\cite[Theorem~4.5]{ah:pgl3} (see also~\cite[Remark~4.6]{ah:pgl3}).  From this, one sees that if $V$ is the costandard $G^{x_C}_\red$-module with simple socle $L$, then $V^{\id \times \sigma_C}$ has simple socle $L^*$, and furthermore it has the same composition factors the appropriate costandard module.  It follows easily from this that $V^{\id \times \sigma_C}$ is in fact isomorphic to the costandard module with simple socle $L^*$.  Dually, if $V$ is a standard module with simple quotient $L$, then $V^{\id \times \sigma_C}$ is the standard module with simple quotient $L^*$.  The claim for tilting modules follows from these observations.
\end{proof}

\begin{lem}
\phantomsection\label{lem:repsilt-duality}
\begin{enumerate}
\item The category $\Silt(C) \subset \Db\CohGGm(C)$ is preserved by the Serre--Grothendieck duality functor.  Specifically, we have\label{it:repsilt-sg}
\[
\D(\cT_\omega[n]\la -n\ra) \cong \cT_\omega^*[-\codim C -n]\la \codim C + n\ra.
\]
\item The category $\Silt(C) \subset \Db\CohGGm(C)$ is preserved by the opposition functor $({-})^{\id \times \sigma_C}: \Db\CohGGm(C) \to \Db\CohGGm(C)$.  Specifically, we have\label{it:repsilt-opp}
\[
(\cT_\omega[n]\la -n\ra)^{\id \times \sigma_C} \cong \cT_\omega^*[n]\la -n\ra.
\]
\end{enumerate}
\end{lem}
\begin{proof}
Part~\eqref{it:repsilt-sg} follows from the observation that under the equivalence~\eqref{eqn:orbit-equiv}, $\D_C$ corresponds to the functor
\[
R\Hom({-},\bk[-\codim C]\la \codim C\ra): \Db\Rep(\Gm \ltimes G^{x_C}) \to \Db\Rep(\Gm \ltimes G^{x_C}).
\]
Part~\eqref{it:repsilt-opp} is an immediate consequence of Lemma~\ref{lem:tilting-oppo}.
\end{proof}

%%%%%%%%%%%%%%%%%%%%%%%%%%%%%%%%%%%%%%%%%%%%%%%%%%%%%%%%%%%%%%%%%%%%%%%%%%%
\section{Nilpotent orbits embedded in the Lie algebra}
\label{sec:lie-extend}
%%%%%%%%%%%%%%%%%%%%%%%%%%%%%%%%%%%%%%%%%%%%%%%%%%%%%%%%%%%%%%%%%%%%%%%%%%%

The goal of this section is to extend the co-$t$-structure~\eqref{eqn:siltc-defn} on a nilpotent orbit $C$ to a co-$t$-structure on infinitesimal neighborhoods of $C$ in $\fg_C$.  More precisely, consider the derived category $\Db_C\CohGGm(\fg_C)$ of complexes of coherent sheaves on $\fg_C$ with set-theoretic support on $C$.  The main result of this section equips $\Db_C\CohGGm(\fg_C)$ with a co-$t$-structure such that
\[
j_{C*}: \Db\CohGGm(C) \to \Db_C\CohGGm(\fg_C)
\]
is co-$t$-exact.

%--------------------------------------------------------------------------
\subsection{Tangent and normal spaces}
%--------------------------------------------------------------------------

We begin with a series of calculations involving the tangent and normal spaces to $C \subset \fg$ at the point $x_C \in C$, denoted by $T_{x_C}C$ and $V_{x_C}$, respectively.  There are $\Gm \times G^{x_C}$-equivariant isomorphisms
\[
T_{x_C}C \cong [x_C,\fg],
\qquad
V_{x_C} \cong \fg/[x_C,\fg].
\]
These spaces fit into the short exact sequence of $\Gm \times G^{x_C}$-modules
\begin{equation}\label{eqn:normal-sequence}
0 \rightarrow [x_C, \fg] \rightarrow \fg \rightarrow \fg/[x_C, \fg] \rightarrow 0.
\end{equation}

\begin{lem}
\phantomsection\label{lem:normal}
\begin{enumerate}
\item As a $\Gm \ltimes G^{x_C}$-module, $V_{x_C}$ has $\Gm$-weights${}\le -2$.\label{it:normal-wts}
\item Let $i \ge 1$. As a $G^{x_C}_\red$-module, $\bigwedge^i V_{x_C}$ has good filtration dimension${}\le i-1$.\label{it:normal-gfd}
\end{enumerate}
\end{lem}
\begin{proof}
The cocharacter $\phi_{x_C}$ induces a grading $\fg = \bigoplus_{i \in \Z} \fg_i$.  The adjoint action of $G^{x_C}_\red$ preserves this grading, and the map $\ad(x): \fg \to \fg$ sends each $\fg_i$ to $\fg_{i+2}$.  Thus,~\eqref{eqn:normal-sequence} is the direct sum of short exact sequences of $G^{x_C}_\red$-modules
\[
0 \to [x_C, \fg_{i-2}] \rightarrow \fg_i \rightarrow \fg_i/[x_C, \fg_{i-2}] \to 0.
\]
According to the proof of~\cite[Proposition~5.8]{jan:nort}, the operator $\ad(x_C): \fg_{i-2} \to \fg_i$ is surjective for $i > 0$, so taking the sum over all $i \le 0$, we obtain
\begin{equation}\label{eqn:normal-sequence-gr}
0 \to \bigoplus_{i \le 0} [x_C, \fg_{i-2}] \to \bigoplus_{i \le 0} \fg_i \to V_{x_C} \to 0.
\end{equation}
Under the isomorphism~\eqref{eqn:semi-split}, we see that $t \in \Gm$ acts on $v \in \fg_i$ by $t \cdot v = (\phi_{x_C}(t),t)\cdot v = \phi_{x_C}(t)t^{-2}v = t^{i-2}v$.  Thus,~\eqref{eqn:normal-sequence-gr} shows that $V_{x_C}$ has $\Gm$-weights${}\le -2$.

By~\cite[Proposition~5.8]{jan:nort} again, $\ad(x_C): \fg_{i-2} \to \fg_i$ is injective for $i -2 < 0$, so we can rewrite~\eqref{eqn:normal-sequence-gr} as short exact sequence of $G^{x_C}_\red$-modules
\begin{equation}\label{eqn:normal-sequence-gr2}
0 \to \fg_{\le -2} \xrightarrow{\ad(x_C)}  \fg_{\le 0} \to V_{x_C} \to 0,
\end{equation}
where we introduce the notation
\[
\fg_{\le n} = \bigoplus_{i \le n} \fg_i.
\]
(Note that the passage from~\eqref{eqn:normal-sequence-gr} to~\eqref{eqn:normal-sequence-gr2} does not respect the $\Gm$-action on the first term.)  This sequence implies that any exterior power $\bigwedge^j \fg_{\le 0}$ admits a $G^{x_C}_\red$-stable filtration
\[
0 = M^j_{-1} \subset M^j_0 \subset M^j_1 \subset \cdots M^j_j = \textstyle\bigwedge^j (\fg_{\le 0})
\]
such that
\[
\textstyle
M^j_k/M^j_{k-1} \cong \bigwedge^{j-k} (\fg_{\le -2}) \otimes \bigwedge^k V_{x_C}.
\]

As a $G^{x_C}_\red$-module, $\fg_{\le n}$ is a direct summand of $\fg$, and so $\bigwedge^j (\fg_{\le n})$ is a direct summand of $\bigwedge^j \fg$.  Using Theorem~\ref{thm:donkin} and Lemma~\ref{lem:wedge-adjoint} below, we conclude that
\begin{equation}\label{eqn:phi-goodfilt}
\text{$\textstyle\bigwedge^j (\fg_{\le n})$ has good filtration dimension${}\le j-1$.}
\end{equation}

We will now show by induction on $j$ (for $j \ge 1$) that every step $M^j_k$ of the filtration described above has good filtration dimension${}\le j-1$.  We will simultaneously prove that $\bigwedge^j V_{x_C}$ has good filtration dimension${}\le j-1$.  If $j = 1$, then the modules
\[
M^1_0 = \fg_{\le -2}
\qquad\text{and}\qquad
M^1_1 = \fg_{\le 0}
\]
have good filtrations by~\eqref{eqn:phi-goodfilt}, and then the short exact sequence~\eqref{eqn:normal-sequence-gr2} shows that $V_{x_C}$ has a good filtration.

Now suppose that $j > 1$.  We start by observing that $M^j_0 = \bigwedge^j \fg_{\le -2}$ and $M^j_j = \bigwedge^j \fg_{\le 0}$ both have good filtration dimension${}\le j-1$ by~\eqref{eqn:phi-goodfilt}.  We treat the remaining $M^j_k$ by induction on $k$. Suppose $0 < k < j$, and that $M^j_{k-1}$ is known to have good filtration dimension${}\le j-1$.  Consider the short exact sequence
\[
\textstyle
0 \to M^j_{k-1} \to M^j_k \to \bigwedge^{j-k} (\fg_{\le -2}) \otimes \bigwedge^k V_{x_C} \to 0.
\]
Since $k < j$, by induction, $\bigwedge^k V_{x_C}$ is known to have good filtration dimension${}\le k-1$, so by~\cite[Proposition~3.4(c)]{fp:claag}, $\bigwedge^{j-k} (\fg_{\le -2}) \otimes \bigwedge^k V_{x_C}$ has good filtration dimension${}\le j-2$.  The short exact sequence above then implies that $M^j_k$ also has good filtration dimension${}\le j-1$.

It remains to show that $\bigwedge^j V_{x_C}$ has good filtration dimension${}\le j-1$.  This follows from the short exact sequence
\[
\textstyle
0 \to M^j_{j-1} \to M^j_j \to \bigwedge^j V_{x_C} \to 0. \qedhere
\]
\end{proof}

\begin{lem}\label{lem:wedge-adjoint}
Let $i \ge 1$. As a $G$-representation, the $i$th exterior power $\bigwedge^i \fg$ of the adjoint representation has good filtration dimension${}\le i-1$.
\end{lem}
\begin{proof}
According to~\cite[Proposition~4.4]{aj}, under our assumptions on $G$, the symmetric algebra $\Sym(\fg)$ has a good filtration.  Recall that $\bigwedge^i\fg$ can be identified with $\Tor_i^{\Sym(\fg)}(\bk,\bk)$.  The latter can be computed using the bar resolution of the trivial $\Sym(\fg)$-module.  Explicitly, if we let
\[
B^i_j = \bigoplus_{\substack{a_1, \ldots, a_j \ge 1\\ a_1 + \cdots + a_j = i}}
\Sym^{a_1}(\fg) \otimes \cdots \otimes \Sym^{a_j}(\fg),
\]
then there is an exact sequence
\[
\textstyle
0 \to \bigwedge^i \fg \to B^i_i \to B^i_{i-1} \to \cdots \to B^i_1 \to 0.
\]
(See, for instance,~\cite[Eqn.~(1.7)]{pri:kr} for a formula for the maps in this complex.)  Since each $B^i_j$ has a good filtration, this sequence shows that $\bigwedge^i \fg$ has good filtration dimension${}\le i-1$.
\end{proof}

\begin{lem}\label{lem:ext-wedge}
Let $j_C: C \hookrightarrow \fg_C$ be the inclusion map.  For all $k \ge 0$, we have $\cExt^k(j_{C*}\cO_C, j_{C*}\cO_C) \cong j_{C*}\bigwedge^k \cV_C$, where $\cV_C$ is the normal bundle on $C$.
\end{lem}
\begin{proof}
Observe first that each $\cExt^k(j_{C*}\cO_C, j_{C*}\cO_C)$ is scheme-theoretically supported on $C$, as shown by the following calculation:
\begin{multline*}
\cExt^k(j_{C*}\cO_C, j_{C*}\cO_C) \cong \cH^k(R\cHom(j_{C*}\cO_C, j_{C*}\cO_C) ) \\
\cong j_{C*}\cH^k(R\cHom(j_C^*j_{C*}\cO_C,\cO_C)).
\end{multline*}
Let $\cE^k \in \Coh^{G \times \Gm}(C)$ be the object such that $\cExt^k(j_{C*}\cO_C, j_{C*}\cO_C) \cong j_{C*}\cE^k$.  

For $k = 0$, it is clear that $\cHom(j_{C*}\cO_C, j_{C*}\cO_C) \cong j_{C*}\cO_C$.  Next, let $\cI \subset \cO_{\fg_C}$ be the ideal sheaf corresponding to $C$, so that we have a short exact sequence $0 \to \cI \to \cO_{\fg_C} \to j_{C*}\cO_C \to 0$.  This gives rise to a long exact sequence
\begin{multline*}
0 \to \cHom(j_{C*}\cO_C, j_{C*}\cO_C) \simto j_{C*}\cO_C \to \cHom(\cI, j_{C*}\cO)\\ \to \cExt^1(j_{C*}\cO_C, j_{C*}\cO_C) \to
 \cExt^1(\cO_{\fg_C}, j_{C*}\cO_C) \to \dots.
\end{multline*}
Since the last term vanishes and the first two terms are isomorphic, we have
\[
\cExt^1(j_{C*}\cO_C, j_{C*}\cO_C) \cong \cHom(\cI, j_{C*}\cO) \cong j_{C*}\cHom(j_C^*\cI, \cO_C).
\]
The sheaf $j_C^*\cI \cong \cI/\cI^2$ is the conormal sheaf, so $\cHom(j_C^*\cI, \cO_C)$ is the normal sheaf $\cV_C$.  So far we have shown that
\begin{equation}\label{eqn:ext-lowdeg}
\textstyle
\cE^0 \cong \cO_C = \bigwedge^0 \cV_C
\qquad\text{and}\qquad
\cE^1 \cong \cV_C = \bigwedge^1 \cV_C.
\end{equation}

To proceed further, we will exploit the fact that $\bigoplus_k \cExt^k(j_{C*}\cO_C, j_{C*}\cO_C)$ is a sheaf of algebras.  It follows that $\bigoplus_k \cE^j$ is also a sheaf of algebras; it corresponds under~\eqref{eqn:orbit-equiv} to a graded ring with a compatible $(G \times \Gm)^{x_C}$-action.   

We will now compute this ring.  Let $S$ be a $\phi_{x_C}$-stable linear complement to $[x_C,\fg]$ in $\fg$.  (In general, $S$ will \emph{not} be stable under $G^{x_C}$.)  Then $x_C + S$ is a transverse slice to $C$ in $\fg$: it does not meet $\partial C$, and the map
\[
m: G \times S \to \fg_C \qquad\text{given by}\qquad (g,s) \mapsto \mathrm{Ad}(g)(x_C+s)
\]
is smooth.  It is also $G \times \Gm$-equivariant, where we let $G \times \Gm$ act on $G \times S$ by
\[
(g,z) \cdot (h,s) = (gh \phi_{x_C}^{-1}(z), z^{-2}\mathrm{Ad}(\phi_{x_C}(z))(s)).
\]
Consider the following diagram:
\[
\begin{tikzcd}
\{0\} \ar[d, "i"'] \ar[r] & G \times \{0\} \ar[r, "m_0"] \ar[d, "\tilde\jmath"'] & C \ar[d, "j_C"] \\
S \ar[r] & G \times S \ar[r, "m"] & \fg_C
\end{tikzcd}
\]
Since $m$ and $m_0$ are smooth, by smooth base change, we have natural isomorphisms
\[
m^*\cExt^k(j_{C*}\cO_C, j_{C*}\cO_C) \cong \cExt^k(\tilde\jmath_*\cO_G, \tilde\jmath_*\cO_G) \cong \tilde \jmath_*m_0^*\cE^k.
\]
Moreover, the direct sum $\bigoplus_k m^*\cExt^k(j_{C*}\cO_C, j_{C*}\cO_C)$ is again a sheaf of algebras.  These sheaves live in $\Coh^{G \times \Gm}(G \times S)$.  Now consider the group $\{(\phi_{x_C}(z),z): z \in \Gm\} \subset G \times \Gm$.  This group is isomorphic to $\Gm$, and it stabilizes $\{e\} \times S \subset G \times S$, so we obtain equivalences of categories
\[
\Coh^{G \times \Gm}(G \times S) \cong \Coh^{\Gm}(S)
\qquad\text{and}\qquad
\Coh^{G \times \Gm}(G) \cong \Coh^{\Gm}(S),
\]
analogous to~\eqref{eqn:orbit-equiv}.  Thus, computing $m^*\cExt^k(j_{C*}\cO_C, j_{C*}\cO_C)$ is equivalent to computing the sheaves $\cExt^k(i_*\cO_{\{0\}}, i_*\cO_{\{0\}}) \cong i_*\Ext^k(i_*\cO_{\{0\}}, i_*\cO_{\{0\}})$. Since $S$ is a vector space, it is well known that 
\[
\textstyle
\bigoplus_{k \ge 0} \Ext^k(i_*\cO_{\{0\}}, i_*\cO_{\{0\}}) \cong \bigwedge^\bullet S.
\]

Let us summarize: the sheaf of rings $\bigoplus_k \cE^k$ corresponds under~\eqref{eqn:orbit-equiv} to a graded $\Gm \ltimes G^{x_C}$-equivariant ring.  The computation above shows that the underlying graded ring is an exterior algebra (the $G^{x_C}$-action is lost in this computation).  In view of~\eqref{eqn:ext-lowdeg}, we must have $\cE^k \cong \bigwedge^k \cV$ for all $k \ge 0$.
\end{proof}

\begin{cor}\label{cor:jc-pushpull}
For any $\cF \in \CohGGm(C)$, we have
\[
\textstyle
\cH^i(j_C^*j_{C*}\cF) \cong \cF \otimes \bigwedge^{-i} \cV_C^*,
\qquad
\cH^i(j_C^!j_{C*}\cF) \cong \cF \otimes \bigwedge^i \cV_C.
\]
\end{cor}
\begin{proof}
We have
\[
j_{C*}R\cHom(j_C^*j_{C*}\cO_C, \cO_C) \cong R\cHom(j_{C*}\cO_C, j_{C*}\cO_C),
\]
and hence, by Lemma~\ref{lem:ext-wedge},
\[
\textstyle
j_{C*}\cH^i(R\cHom(j_C^*j_{C*}\cO_C, \cO_C)) \cong \cExt^i(j_{C*}\cO_C, j_{C*}\cO_C) \cong j_{C*}\bigwedge^i \cV_C.
\]
The functor $\cHom({-},\cO_C)$ on $\CohGGm(C)$ is exact; it corresponds under ~\eqref{eqn:orbit-equiv} to taking the contragredient of a $(G \times \Gm)^{x_C}$-represesentation.  We conclude that $\cHom(\cH^{-i}(j_C^*j_{C*}\cO_C), \cO_C) \cong \bigwedge^i \cV_C$, and therefore
\[
\textstyle
\cH^i(j_C^*j_{C*}\cO_C) \cong \bigwedge^{-i} \cV_C^*.
\]

Now let $\cF \in \CohGGm(C)$.  It is enough to prove that the isomorphisms in the statement of the corollary hold after applying $j_{C*}$.  The projection formula implies that
\[
j_{C*}j_C^*j_{C*}\cF \cong j_{C*}(\cO_C \lotimes j_C^*j_{C*}\cF) 
\cong j_{C*}\cO_C \lotimes j_{C*}\cF \cong j_{C*}(j_C^*j_{C*}\cO_C \lotimes \cF).
\]
Since $\otimes$ on $\CohGGm(C)$ is exact, we deduce that
\[
\textstyle
\cH^i(j_C^*j_{C*}\cF) \cong \cH^i((j_C^*j_{C*}\cO_C) \otimes \cF \cong \bigwedge^{-i} \cV_C^* \otimes \cF.
\]
The formula for $j_C^!j_{C*}\cF$ follows by a similar calculation using
\[
j_{C*}j_C^!j_{C*}\cF \cong j_{C*}R\cHom(\cO_C, j_C^!j_{C*}\cF)
\cong j_{C*}R\cHom(j_C^*j_{C*}\cO_C, \cF). \qedhere
\]
\end{proof}

%--------------------------------------------------------------------------
\subsection{Construction of the co-\texorpdfstring{$t$}{t}-structure}
%--------------------------------------------------------------------------

We are now ready to put the calculations above to use.

\begin{lem}
\phantomsection\label{lem:jc-unit-exact}
\begin{enumerate}
\item The functor $j_C^!j_{C*}: \Db\CohGGm(C) \to \Db\CohGGm(C)$ is left co-$t$-exact.  Moreover, for $\cF \in \Db\CohGGm(C)_{\le 0}$, the cone of the adjunction map $\cF \to j^!j_*\cF$ lies in $\Db\CohGGm(C)_{\le -1}$.
\item The functor $j_C^*j_{C*}: \Db\CohGGm(C) \to \Db\CohGGm(C)$ is right co-$t$-exact.  Moreover, for $\cF \in \Db\CohGGm(C)_{\ge 0}$, the cocone of the adjunction map $j^*j_*\cF \to \cF$ lies in $\Db\CohGGm(C)_{\ge 1}$.
\end{enumerate}
\end{lem}
\begin{proof}
We will prove the first assertion; the second one is similar.  We must show that if $\cF \in \Db\CohGGm(C)_{\le 0}$, then $j_C^!j_{C*}\cF \in \Db\CohGGm(C)_{\le 0}$.  Since $\Db\CohGGm(C)$ is generated under extensions by objects of the form $\cT_\omega[n]\la k\ra$ with $n+k \ge 0$, it is enough to consider the special case $\cF = \cT_\omega[n]\la k\ra$.  The adjunction map $\cF \to j_C^!j_{C*}\cF$ induces an isomorphism $\cF \simto \cH^0(j_C^!j_{C*}\cF)$.  Thus, to prove the lemma, it is enough to show that the higher cohomology sheaves $\cH^i(j_C^!j_{C*}\cF)[-i]$ with $i \ge 1$ lie in $\Db\CohGGm(C)_{\le -1}$.  By  Corollary~\ref{cor:jc-pushpull}, we have
\[
\textstyle
\cH^{i-n}(j_C^!j_{C*}\cT_\omega[n]\la k\ra) \cong \cT_\omega\la k\ra \otimes \bigwedge^i\cV_C.
\]
Thus, the lemma commes down to showing that
\begin{equation}\label{eqn:jc-unit-rep}
\textstyle
\tilt_\omega\la k\ra \otimes \bigwedge^i V_{x_C}[n-i] \in \Db\Rep(\Gm \ltimes G^{x_C})_{\le -1}
\qquad\text{for $i \ge 1$.}
\end{equation}

We will now prove~\eqref{eqn:jc-unit-rep}.  By Lemma~\ref{lem:normal}\eqref{it:normal-wts}, the $\Gm$-action on $\bigwedge^i V_{x_C}$ has weights${}\le -2i$.  Therefore, as a $\Gm \times G^{x_C}_\red$-representation, $\bigwedge^i V_{x_C}$ can be decomposed as
\[
\textstyle
\bigwedge^i V_{x_C} = \bigoplus_{j \ge 2i} N_{ij}\la j\ra
\]
where each $N_{ij}$ is some $G^{x_C}_\red$-module, regarded as a $\Gm \times G^{x_C}_\red$-module with trivial $\Gm$-action.  

By Lemma~\ref{lem:normal}\eqref{it:normal-gfd}, each $N_{ij}$ has good filtration dimension${}\le i-1$, and hence so does $\tilt_\omega \otimes N_{ij}$ (cf.~\cite[Proposition~3.4]{fp:claag}).  By Lemma~\ref{lem:gfd-co-t}, we have $\tilt_\omega\la k\ra \otimes N_{ij}\la j\ra [n-i] \in \Db\Rep(\Gm \ltimes G^{x_C})_{\le i -1 - (j+k) - (n-i)}$.  Since $j \ge 2i$ and $n+k \ge 0$, this object lies in $\Db\Rep(\Gm \ltimes G^{x_C})_{\le -1}$, as desired.
\end{proof}

\begin{prop}\label{prop:lie-extend}
For any nilpotent orbit $C \subset \fg$, there is a unique co-$t$-structure on $\Db_C\CohGGm(\fg_C)$ such that
\[
j_{C*}: \Db\CohGGm(C) \to \Db_C\CohGGm(\fg_C)
\]
is co-$t$-exact.  The indecomposable silting objects in $\Db_C\CohGGm(\fg_C)$ are precisely those of the form $j_{C*}\cT_\omega[n]\la -n\ra$ with $\omega \in \Omega_C$ and $n \in \Z$.
\end{prop}
\begin{proof}
We wish to show that objects of the form $j_{C*}\cT$ with $\cT \in \Silt(C)$ form a silting subcategory of $\Db_C\CohGGm(\fg_C)$.  These objects clearly remain indecomposable and generate $\Db_C\CohGGm(\fg_C)$, so it remains to show that for any $\cT, \cT' \in \Silt(C)$, we have
\[
\Hom(j_{C*}\cT, j_{C*}\cT'[n]) = 0 \qquad\text{whenever $n > 0$.}
\]
This follows by adjunction and Lemma~\ref{lem:jc-unit-exact}.
\end{proof}

By construction, the co-$t$-structure obtained in Proposition~\ref{prop:lie-extend} has the following explicit description:
\begin{align*}
\Db_C\CohGGm(\fg_C)_{\ge 0} &=
\begin{array}{@{}c@{}}
\text{the full subcategory of $\Db_C\CohGGm(\fg_C)$ generated} \\
\text{under extensions by $j_{C*}\cT_\omega[n]\la k\ra$ with $n + k \le 0$}
\end{array} \\
\Db_C\CohGGm(\fg_C)_{\le 0} &=
\begin{array}{@{}c@{}}
\text{the full subcategory of $\Db_C\CohGGm(\fg_C)$ generated} \\
\text{under extensions by $j_{C*}\cT_\omega[n]\la k\ra$ with $n + k \ge 0$}.
\end{array}
\end{align*}
The coheart of this co-$t$-structure is denoted by
\[
\Silt_C(\fg_C) = \Db_C\CohGGm(\fg_C)_{\ge 0} \cap \Db_C\CohGGm(\fg_C)_{\le 0}.
\]
One can also describe this co-$t$-structure using vector bundles corresponding to Weyl or dual Weyl modules for $G^{x_C}_\red$, in analogy with Remark~\ref{rmk:weyl-cot}.

\begin{rmk}\label{rmk:lie-equiv}
One can show that the functor $j_{C*}: \Silt(C) \to \Silt_C(\fg_C)$ is full: see Lemma~\ref{lem:silt-full} below for a similar argument.
\end{rmk}

\begin{lem}\label{lem:lie-duality}
The category $\Silt_C(\fg_C) \subset \Db_C\CohGGm(\fg_C)$ is preserved by the Serre--Grothendieck duality functor.  Specifically, we have
\[
\D(j_{C*}\cT_\omega[n]\la -n\ra) \cong j_{C*}\cT_\omega^*[-\codim C -n]\la \codim C + n\ra.
\]
\end{lem}
\begin{proof}
This is immediate from~\eqref{eqn:serre-groth-open} and Lemma~\ref{lem:repsilt-duality}.
\end{proof}

%%%%%%%%%%%%%%%%%%%%%%%%%%%%%%%%%%%%%%%%%%%%%%%%%%%%%%%%%%%%%%%%%%%%%%%%%%%
\section{Nilpotent orbits embedded in the nilpotent cone}
\label{sec:nil-extend}
%%%%%%%%%%%%%%%%%%%%%%%%%%%%%%%%%%%%%%%%%%%%%%%%%%%%%%%%%%%%%%%%%%%%%%%%%%%

Let $C \subset \fg$ be a nilpotent orbit, and let $U \subset \cN$ be a $G$-stable open subset such that $C$ is closed as a subset of $U$.  The goal of this section is to show that the co-$t$-structure~\eqref{eqn:siltc-defn} on  $C$ extends to a co-$t$-structure on infinitesimal neighborhoods of $C$ in $U$.  

Throughout this section, we let
\[
i_C: C \hookrightarrow U
\qquad\text{and}\qquad
j: U \hookrightarrow \fg_C
\]
be the inclusion maps.  Here, $i_C$ is a closed immersion, and $j$ is a locally closed immersion.  In general, $j_*$ takes values in the derived category $D^+\QCoh^{G \times \Gm}(\fg_C)$ of \emph{quasicoherent} sheaves on $\fg_C$.  However, in the important special case where $U = \cN_C$, $j$ is a closed immersion, and $j_*$ sends $\Db\CohGGm(\cN_C)$ to $\Db\CohGGm(\fg_C)$.

\begin{lem}\label{lem:complete-int}
There is a collection of positive integers $n_1, \ldots, n_r$ such that for any $\cF \in \CohGGm(U)$, we have
\begin{align*}
\cH^k(j^*j_*\cF) &\cong \bigoplus_{\{i_1, \ldots, i_k\} \subset \{1, \ldots, r\}} \cF\la -2(n_{i_1} + \cdots + n_{i_k})\ra, \\
\cH^k(j^!j_*\cF) &\cong \bigoplus_{\{i_1, \ldots, i_k\} \subset \{1, \ldots, r\}} \cF\la 2(n_{i_1} + \cdots + n_{i_k})\ra.
\end{align*}
\end{lem}
\begin{proof}
Since any $\cF \in \CohGGm(U)$ can be extended to an object of $\CohGGm(\cN)$, it is enough to prove the lemma in the special case where $C$ is the zero orbit, so that $U = \cN$.  We will work on $\cN$ from now on, and let $j: \cN \hookrightarrow \fg$ be the inclusion map.

The $\Gm$-action on $\fg$ equips its coordinate ring $\bk[\fg] = \Sym(\fg^*)$ with a grading in nonnegative, even degrees.  As in \cite[\S 7.13]{jan:nort}, let $f_1,\dots, f_r$ denote a minimal set of homogeneous generators of the $\bk$-subalgebra $\Sym(\fg^*)^G$.  These generators have strictly positive, even degrees. By \cite[Proposition~7.13]{jan:nort}, we have $r = \codim_{\fg}(\cN)$ and 
\[
\bk[\cN] \cong \Sym(\fg^*)/\Sym(\fg^*)\la f_1,\dots, f_r\ra
\]
 as a $G\times \Gm$-equivariant algebra.  As a consequence, $\cN$ is a complete intersection. Let $V = \la f_1,\dots, f_r \ra$ denote their
 $\bk$-span, and note that $V$ is trivial as a $G$-module and has strictly positive $\Gm$-weights, thus as a $G\times \Gm$-module, 
 \[
 V \cong \bigoplus_{j=1,\dots,r} \bk\la-2n_j\ra,
 \] 
where $n_j = \frac{1}{2}\deg f_j \ > 0$ for all $j$.

There is a $G\times \Gm$-equivariant Koszul resolution 
\begin{equation}\label{eqn:complete-int-res}
 \textstyle
   0 \rightarrow \cO_\fg \otimes \bigwedge^r(V) \rightarrow \cdots  \rightarrow  \cO_\fg\otimes \bigwedge^2(V)  \rightarrow  \cO_\fg\otimes V\rightarrow  \cO_\fg \rightarrow  j_*\cO_\cN \rightarrow 0.
 \end{equation}

As in the proof of Corollary~\ref{cor:jc-pushpull}, to prove the present lemma, it is enough to compute the cohomology sheaves of the objects
\[
j_*j^*j_*\cF \cong j_*\cO_\cN \lotimes j_*\cF
\qquad\text{and}\qquad
j_*j^!j_*\cF \cong R\cHom(j_*\cO_\cN, j_*\cF).
\]
Both of these can computed using the resolution~\eqref{eqn:complete-int-res}.  The differentials in that resolution are defined in terms of multiplication by one of the $f_i$'s, which vanish on $\cN$.  Thus, after applying $({-}) \otimes j_*\cF$ or $\cHom({-},j_*\cF)$, the differentials become zero.  We conclude that
\[
\textstyle
\cH^i(j_*j^*j_*\cF) \cong \bigwedge^{-i} V \otimes j_*\cF,
\qquad
\cH^i(j_*j^!j_*\cF) \cong \bigwedge^i V^* \otimes j_*\cF,
\]
and the result follows.
\end{proof}

\begin{prop}\label{prop:nil-extend}
Let $U \subset \cN$ be a $G$-stable open subset, and let $i_C: C \hookrightarrow U$ be the inclusion of a nilpotent orbit that is closed in $U$.  There is a unique co-$t$-structure on $\Db_C\CohGGm(U)$ such that
\[
i_{C*}: \Db\CohGGm(C) \to \Db_C\CohGGm(U)
\]
is co-$t$-exact.  The indecomposable silting objects in $\Db_C\CohGGm(U)$ are precisely those of the form $i_{C*}\cT_\omega[n]\la -n\ra$ with $\omega \in \Omega_C$ and $n \in \Z$.
\end{prop}
\begin{proof}
As in the proof of Proposition~\ref{prop:lie-extend}, we will be done if we can show that for any $\omega$, $\upsilon \in \Omega_C$, 
\begin{equation}\label{eqn:tilt-hom}
\Hom(i_{C*}\cT_\omega, i_{C*}\cT_\upsilon\langle k \rangle[n]) = 0\qquad\text{if $n > -k$.}
\end{equation}
 We have the maps $i_C: C \hookrightarrow U$, $j: U \hookrightarrow \fg_C$ and 
$j_C = j \circ i_C: C \hookrightarrow \fg_C$. 
We can already deduce from Proposition~\ref{prop:lie-extend} that 
\begin{equation}\label{eqn:push-pull-vanish}
\Hom(j^*j_*i_{C*}\cT_\omega, i_{C*}\cT_\upsilon\langle k \rangle[n]) = \Hom(j_{C*}\cT_\omega, j_{C*}\cT_\upsilon\langle k \rangle[n]) =0 \qquad\text{if $n > -k$.}
\end{equation}
Let $\cK = \tau^{\le -1} j^*j_*i_{C*}\cT_\omega$.  Since $\cH^0(j^*j_*i_{C*}\cT_\omega) \cong i_{C*}\cT_\omega$, we have a truncation distinguished triangle
\begin{equation}\label{eqn:trunc-dt}
\cK \to j^*j_*i_{C*}\cT_\omega \to i_{C*} \cT_\omega \to.
\end{equation}
Lemma~\ref{lem:complete-int} implies that for $1 \leq i \leq r$, 
\begin{equation}\label{eqn:trunc-dt-coh}
\cH^{-i}(\cK) = \cH^{-i}(j^*j_*i_{C*}\cT_\omega) \cong \bigoplus_{j =1}^{\binom{r}{i}} i_{C*}\cT_\omega\la -n_{ij} \ra,
\end{equation}
where $n_{ij} \geq 2i$ for all $i$, $j$.  

To prove~\eqref{eqn:tilt-hom}, suppose instead that this $\Hom$-group is nonzero for some integers $k$, $n$ with $n > -k$. Moreover, assume that our pair $(k,n)$ is chosen with $n$ minimal.  (This is possible since this $\Hom$-group automatically vanishes for $n < 0$.)  Apply $\Hom({-}, i_{C*}\cT_\upsilon\langle k \rangle[n])$ to~\eqref{eqn:trunc-dt} to obtain a long exact sequence
\begin{multline*}
\dots \to \Hom(\cK, i_{C*}\cT_\upsilon\langle k \rangle[n-1]) \xrightarrow{f} \Hom(i_{C*}\cT_\omega, i_{C*}\cT_\upsilon\langle k \rangle[n])  \\
\to \Hom(j^*j_*i_{C*}\cT_\omega, i_{C*}\cT_\upsilon\langle k \rangle[n])  \to \dots.
\end{multline*}
The last term vanishes by~\eqref{eqn:push-pull-vanish}, so $f$ is surjective, and hence $\Hom(\cK, i_{C*}\cT_\upsilon\langle k \rangle[n-1]) \ne 0$.

On the other hand, for $i = 1,\ldots, r$, we have
\[
\Hom(\cH^{-i}(\cK)[i], i_{C*}\cT_\upsilon\langle k\rangle[n-1])
\cong \bigoplus_{j=1}^{\binom{r}{i}} \Hom(i_{C*}\cT_\omega, i_{C*}\cT_\upsilon\langle k + n_{ij}\rangle[n-1-i]).
\]
Since $n_{ij} \ge 2i \ge i+1$ for every term here, we have $(n-1-i) > -(k+n_{ij})$.  If one of the $\Hom$-groups above is nonzero, that would contradict the minimality of $n$ in our pair $(k,n)$.  So we must have $\Hom(\cH^{-i}(\cK)[i], i_{C*}\cT_\upsilon\langle k\rangle[n-1]) = 0$ for all $i$.  From this, it is easily deduced that $\Hom(\cK, i_{C*}\cT_\upsilon\langle k\rangle[n-1]) = 0$, contradicting the previous paragraph.
\end{proof}

As usual, the co-$t$-structure obtained above can be described as follows:
\begin{align*}
\Db_C\CohGGm(U)_{\ge 0} &=
\begin{array}{@{}c@{}}
\text{the full subcategory of $\Db_C\CohGGm(U)$ generated} \\
\text{under extensions by $i_{C*}\cT_\omega[n]\la k\ra$ with $n + k \le 0$}
\end{array} \\
\Db_C\CohGGm(U)_{\le 0} &=
\begin{array}{@{}c@{}}
\text{the full subcategory of $\Db_C\CohGGm(U)$ generated} \\
\text{under extensions by $i_{C*}\cT_\omega[n]\la k\ra$ with $n + k \ge 0$}
\end{array}
\end{align*}
The coheart of this co-$t$-structure is denoted by
\[
\Silt_C(U) = \Db_C\CohGGm(U)_{\ge 0} \cap \Db_C\CohGGm(U)_{\le 0}.
\]

\begin{lem}\label{lem:nilp-duality}
The category $\Silt_C(U) \subset \Db_C\CohGGm(U)$ is preserved by the Serre--Grothendieck duality functor.  Specifically, we have
\[
\D(i_{C*}\cT_\omega[n]\la -n\ra) \cong i_{C*}\cT_\omega^*[-\codim C -n]\la \codim C + n\ra.
\]
\end{lem}
\begin{proof}
This is immediate from~\eqref{eqn:serre-groth-orbit} and Lemma~\ref{lem:repsilt-duality}.
\end{proof}

We can extract the following corollary from the proof of Proposition~\ref{prop:nil-extend}.

\begin{cor}
\phantomsection\label{cor:j-adj-strict}
\begin{enumerate}
\item If $\cF \in \Db_C\CohGGm(U)_{\le 0}$, then the cone of the adjunction map $\cF \to j^!j_*\cF$ lies in $\Db_C\CohGGm(U)_{\le -1}$.
\item If $\cF \in \Db_C\CohGGm(U)_{\ge 0}$, then the cocone of the adjunction map $j^*j_*\cF \to \cF$ lies in $\Db_C\CohGGm(U)_{\ge 1}$.
\end{enumerate}
\end{cor}
\begin{proof}
We will prove the second assertion; the first one is similar.  It is enough to consider the special case $\cF = i_{C*}\cT_\omega[n]\la k\ra$ with $n+k \le 0$.  We most show that the object $\cK$ in~\eqref{eqn:trunc-dt} lies in $\Db_C\CohGGm(U)_{\ge 1}$.  This follows from~\eqref{eqn:trunc-dt-coh} and the fact that $-i-n_{ij} \le -1$ for all $i \ge 1$.
\end{proof}

The next two statements involve the special case $U = \cN_C$, where $j$ is a closed immersion.

\begin{lem}\label{lem:silt-full}
If $\cF \in \Db_C\CohGGm(\cN_C)_{\ge 0}$ and $\cG \in \Db_C\CohGGm(\cN_C)_{\le 0}$, then the natural map
\[
\Hom(\cF,\cG) \to \Hom(j_*\cF, j_*\cG)
\]
is surjective.  In particular, $j_*$ restricts to a full functor $\Silt_C(\cN_C) \to \Silt_C(\fg_C)$.
\end{lem}
\begin{proof}
Let $\cK$ be the cone of the adjunction map $\cG \to j^!j_*\cG \to \cK \to$.  We have a long exact sequence
\[
\cdots \to \Hom(\cF,\cG) \to \Hom(\cF, j^!j_*\cG) \to \Hom(\cF,\cK) \to \cdots.
\]
By Corollary~\ref{cor:j-adj-strict}, we have $\cK \in \Db_C\CohGGm(\cN_C)_{\le -1}$, so the last term vanishes.  The lemma follows.
\end{proof}

\begin{cor}\label{cor:jc-co-t-exact}
The functor
\[
j_*: \Db_C\CohGGm(\cN_C) \to \Db_C\CohGGm(\fg_C)
\]
is co-$t$-exact.  The functors
\[
j^*, j^!: \Db_C\CohGGm(\fg_C) \to \Db_C\CohGGm(\cN_C)
\]
are right and left co-$t$-exact, respectively.
\end{cor}
\begin{proof}
The claim about $j_*$ is obvious from the description of silting objects in the two categories.  The claims for $j^*$ and $j^!$ follow by adjunction.
\end{proof}

%%%%%%%%%%%%%%%%%%%%%%%%%%%%%%%%%%%%%%%%%%%%%%%%%%%%%%%%%%%%%%%%%%%%%%%%%%%
\section{The nilpotent cone}
\label{sec:nilcone}
%%%%%%%%%%%%%%%%%%%%%%%%%%%%%%%%%%%%%%%%%%%%%%%%%%%%%%%%%%%%%%%%%%%%%%%%%%%

This section contains the main geometric result of the paper: the construction of a co-$t$-structure on $\Db\CohGGm(\cN)$.  We will build this co-$t$-structure using the co-$t$-structures on infinitesimal neighborhoods of nilpotent orbits from Proposition~\ref{prop:nil-extend}.  More generally, we obtain a co-$t$-structure on $\Db\CohGGm(U)$ for any $G$-stable open subset $U \subset \cN$.

As a technical tool, we will use the following full subcategories of $\Db_{\cN_C}(\fg_C)$:
\begin{align*}
^\perp\Db_C\CohGGm(\fg_C)_{\le n} &= \left\{ \cF \in \Db_{\cN_C}\CohGGm(\fg_C) \,\Big|\, 
\begin{array}{@{}c@{}}
\text{for $\cG \in \Db_C\CohGGm(\fg_C)_{\le n}$,} \\
\text{we have $\Hom(\cF,\cG) = 0$}
\end{array} \right\}, \\
\Db_C\CohGGm(\fg_C)_{\ge n}^\perp &= \left\{ \cF \in \Db_{\cN_C}\CohGGm(\fg_C) \,\Big|\, 
\begin{array}{@{}c@{}}
\text{for $\cG \in \Db_C\CohGGm(\fg_C)_{\ge n}$,} \\
\text{we have $\Hom(\cG,\cF) = 0$}
\end{array} \right\}.
\end{align*}
Note that objects of $^\perp\Db_C\CohGGm(\fg_C)_{\le n}$ or $\Db_C\CohGGm(\fg_C)_{\ge n}^\perp$ are \emph{not} required to have set-theoretic support on $C$: they can be supported on all of $\cN_C$.

Similarly, for any $G$-stable open subset $U \subset \cN$ that contains $C$ as a closed subset, we define
\begin{align*}
^\perp\Db_C\CohGGm(U)_{\le n} &= \left\{ \cF \in \Db\CohGGm(U) \,\Big|\, 
\begin{array}{@{}c@{}}
\text{for $\cG \in \Db_C\CohGGm(U)_{\le n}$,} \\
\text{we have $\Hom(\cF,\cG) = 0$}
\end{array} \right\}, \\
\Db_C\CohGGm(U)_{\ge n}^\perp &= \left\{ \cF \in \Db\CohGGm(U) \,\Big|\, 
\begin{array}{@{}c@{}}
\text{for $\cG \in \Db_C\CohGGm(U)_{\ge n}$,} \\
\text{we have $\Hom(\cG,\cF) = 0$}
\end{array} \right\}.
\end{align*}
An important special case of these categories is that in which $U =\cN_C$.

\begin{lem}\label{lem:perp-crit}
Let $\cF \in \Db_C\CohGGm(\cN_C)$, and let $U \subset \cN$ be a $G$-stable open subset containing $C$ as a closed subset.  Then $\cF$ lies in $^\perp\Db_C\CohGGm(\cN_C)_{\le n}$, resp.~$\Db_C\CohGGm(\cN_C)_{\ge n}^\perp$, if and only if the object $\cF|_U$ lies in  $^\perp\Db_C\CohGGm(U)_{\le n}$, resp.~$\Db_C\CohGGm(U)_{\ge n}^\perp$.
\end{lem}
\begin{proof}
If $\cG \in \Db_C\CohGGm(\cN_C)$, then the support of $\cG$ is contained in $C \subset U$, so $\Hom(\cF,\cG) \cong \Hom(\cF|_U, \cG|_U)$ and $\Hom(\cG,\cF) \cong \Hom(\cG|_U, \cF|_U)$.
\end{proof}

For the next lemma, assume that $U = \cN_C$.  The proof is identical to that of Lemma~\ref{lem:silt-full} and will be omitted.

\begin{lem}\label{lem:silt-full-2}
If $\cF \in {}^\perp\Db_C\CohGGm(\cN_C)_{\le -1}$ and $\cG \in \Db_C\CohGGm(\cN_C)_{\le 0}$, then the natural map $\Hom(\cF,\cG) \to \Hom(j_*\cF, j_*\cG)$ is surjective.  
\end{lem}

\begin{lem}
\phantomsection\label{lem:pregluing}
\begin{enumerate}
\item For all $\cF \in \Db_{\cN_C}\CohGGm(\fg_C)$, there exist integers $a < b$ such that\label{it:gc-bdd}
\[
\cF \in {}^\perp\Db_C\CohGGm(\fg_C)_{\le a} \cap \Db_C\CohGGm(\fg_C)_{\ge b}^\perp.
\]
\item For $\cF \in \Db\CohGGm(\cN_C)$, we have
$\cF \in \Db_C\CohGGm(\cN_C)_{\ge b}^\perp$
if and only if
$j_*\cF \in \Db_C\CohGGm(\fg_C)_{\ge b}^\perp$.\label{it:gc-geqp}
\item For $\cF \in \Db\CohGGm(\cN_C)$, we have
$\cF \in {}^\perp\Db_C\CohGGm(\cN_C)_{\le a}$
if and only if
$j_*\cF \in {}^\perp\Db_C\CohGGm(\fg_C)_{\le b}$.\label{it:gc-leqp}
\end{enumerate}
\end{lem}
\begin{proof}
\eqref{it:gc-bdd}~Given $\cF \in \Db_{\cN_C}\CohGGm(\fg_C)$, any map $\cF \to j_{C*}\cT_\omega[n]\la k\ra$ factors through $\cF \to j_{C*}j_C^*\cF$.  The object $j_{C*}j_C^*\cF$ lives in $\Db_C\CohGGm(\fg_C)$.  With respect to the bounded co-$t$-structure on that category, we have
\[
j_{C*}j_C^*\cF \in \Db_C\CohGGm(\fg_C)_{\ge a+1}
\]
for some integer $a$.  It follows that $\Hom(\cF,j_{C*}\cT_\omega[n]\la k\ra) = 0$ if $n+k \ge -a$, and hence that $\cF \in {}^\perp\Db_C\CohGGm(\fg_C)_{\le a}$.  The proof of the existence of an integer $b$ such that $\cF \in \Db_C\CohGGm(\fg_C)_{\ge b}^\perp$ is similar, using $j_{C*}j_C^!\cF \to \cF$.

\eqref{it:gc-geqp}~If $\cF \in \Db_C\CohGGm(\cN_C)_{\ge b}^\perp$, then  $j_*\cF \in \Db_C\CohGGm(\fg_C)_{\ge b}^\perp$ by adjunction and the right co-$t$-exactness of $j^*: \Db_C\CohGGm(\fg_C) \to \Db_C\CohGGm(\cN_C)$ (see Corollary~\ref{cor:jc-co-t-exact}).  For the opposite implication, suppose $j_*\cF \in \Db_C\CohGGm(\fg_C)_{\ge b}^\perp$ but $\cF \notin \Db_C\CohGGm(\cN_C)_{\ge b}^\perp$.  For simplicity, let us assume without loss of generality that $b = 0$.  Then there is some nonzero morphism $i_{C*}\cT_\omega \to \cF[n]\la k\ra$ with $n+k > 0$.  Choose such a morphism with $n$ minimal.  We will now follow the pattern of the proof of Proposition~\ref{prop:nil-extend}.  We consider the distinguished triangle~\eqref{eqn:trunc-dt}, which gives rise to a long exact sequence
\begin{multline*}
\cdots \to \Hom(\cK[1], \cF[n]\la k\ra) \to
\Hom(i_{C*}\cT_\omega, \cF[n]\la k\ra) \\
\to \Hom(j^*j_{C*}\cT_\omega, \cF[n]\la k\ra) \to \cdots.
\end{multline*}
The last term vanishes by adjunction and the fact that $j_*\cF \in \Db_C\CohGGm(\fg_C)_{\ge 0}^\perp$, so the first term must be nonzero.  By the same reasoning as in Proposition~\ref{prop:nil-extend}, this implies that
\[
\Hom(i_{C*}\cT_\omega, \cF\la k+n_{ij}\ra[n-1-i]) \ne 0
\]
for some integers $i \ge 1$ and $n_{ij} \ge i+1$, but this contradicts the minimality of $n$.  

\eqref{it:gc-leqp}~The proof of this statement is very similar to that of part~\eqref{it:gc-geqp}.  We omit the details.
\end{proof}

By adjunction, we have
\begin{equation}\label{eqn:pregluing-crit}
\begin{aligned}
\cF \in {}^\perp\Db_C\CohGGm(\fg_C)_{\le a} \quad&\text{if and only if} \quad j_C^*\cF \in \Db\CohGGm(C)_{\ge a+1}, \\
\cF \in \Db_C\CohGGm(\fg_C)_{\ge b}^\perp \quad&\text{if and only if} \quad j_C^!\cF \in \Db\CohGGm(C)_{\le b-1}.
\end{aligned}
\end{equation}

\begin{lem}\label{lem:pregluing-bdd}
For all $\cF \in \Db\CohGGm(U)$, there exist integers $a < b$ such that
\[
\cF \in {}^\perp\Db_C\CohGGm(U)_{\le a} \cap \Db_C\CohGGm(U)_{\ge b}^\perp.
\]
\end{lem}
\begin{proof}
Choose some $\cF' \in \Db\CohGGm(\cN_C)$ such that $\cF'|_U \cong \cF$.  By Lemma~\ref{lem:perp-crit}, it is enough to show that there exist integers $a < b$ such that
\[
\cF' \in {}^\perp\Db_C\CohGGm(\cN_C)_{\le a} \cap \Db_C\CohGGm(\cN_C)_{\ge b}^\perp.
\]
This claim is immediate from Lemma~\ref{lem:pregluing}.
\end{proof}

\begin{prop}\label{prop:preglue-extend}
Let $U \subset \cN$ be a $G$-stable open subset that contains $C$ as a closed subset.  Let $V = U \smallsetminus C$, and let $\cF \in \Db\CohGGm(V)$.  Then there exists an object $\tilde\cF \in \Db\CohGGm(U)$ such that
\[
\tilde\cF|_V \cong \cF
\qquad\text{and}\qquad \tilde\cF \in {}^\perp\Db_C\CohGGm(U)_{\le -1} \cap \Db_C\CohGGm(U)_{\ge 1}^\perp.
\]
Moreover, if $\cF$ is indecomposable, then $\tilde\cF$ can be chosen to be indecomposable as well.
\end{prop}
\begin{proof}
For the existence of $\tilde \cF$, Lemma~\ref{lem:perp-crit} implies that it is enough to work in the special case where $U = \cN_C$.  We will assume that $U = \cN_C$ until the last paragraph of the proof.  Under this assumption, Corollary~\ref{cor:jc-co-t-exact} and Lemma~\ref{lem:silt-full-2} are available.

Choose some object $\cF' \in \Db\CohGGm(\cN_C)$ such that $\cF'|_V \cong \cF$, and let $a \le b$ be integers
\begin{equation}\label{eqn:preglue-fund}
\cF' \in {}^\perp\Db_C\CohGGm(\cN_C)_{\le a} \cap \Db_C\CohGGm(\cN_C)_{\ge b}^\perp.
\end{equation}
Of course, $a$ may be replaced by any smaller integer, and $b$ by any larger integer.  We may therefore assume that $a \le -1$ and $b \ge 1$.

Suppose for now that $a < -1$. By~\eqref{eqn:pregluing-crit}, we have $j_C^*j_*\cF' \in \Db\CohGGm(C)_{\ge a+1}$, so by the axioms for a co-$t$-structure, there is a distinguished triangle
\[
\cG_1 \to j_C^*j_*\cF' \to \cG_2 \to
\]
where $\cG_1 \in \Db\CohGGm(C)_{\ge a+2}$, and
\[
\cG_2 \in \Db\CohGGm(C)_{\ge a+1} \cap \Db\CohGGm(C)_{\le a+1}.
\]
Now let $\phi:j_*\cF' \to j_*i_{C*}\cG_2$ be the composition
\[
j_*\cF' \to j_{C*}j_C^*j_*\cF' \to j_{C*}\cG_2 = j_*i_{C*}\cG_2,
\]
where the first map is an adjunction map.  We claim that if $n+k = -a-1$, then the map
\begin{equation}\label{eqn:fpp-surj}
\Hom(j_{C*}\cG_2, j_{C*}\cT_\omega[n]\la k\ra) \to 
\Hom(j_*\cF', j_{C*}\cT_\omega[n]\la k\ra)
\end{equation}
induced by $\phi$ is surjective.  Indeed, by adjunction, any map $j_*\cF' \to j_{C*}\cT_\omega[n]\la k\ra$ factors through $j_*\cF' \to j_{C*}j_C^*j_*\cF'$, and then the long exact sequence
\begin{multline*}
\cdots \to
\Hom(j_{C*}\cG_2, j_{C*}\cT_\omega[n]\la k\ra) \to
\Hom(j_{C*}j_C^*j_*\cF', j_{C*}\cT_\omega[n]\la k\ra) \\
\to \Hom(j_{C*}\cG_1, j_{C*}\cT_\omega[n]\la k\ra)
\to \cdots
\end{multline*}
proves the claim, because the last term vanishes.

By Lemma~\ref{lem:silt-full-2}, the map $\phi$ is obtained by applying $j_*$ to some (not necessarily unique) morphism $\tilde\phi: \cF' \to i_{C*}\cG_2$.  Complete this map to a distinguished triangle
\begin{equation}\label{eqn:preglue-dt}
\cF'' \to \cF' \xrightarrow{\tilde\phi} i_{C*}\cG_2 \to.
\end{equation}
Since $a+1 < b-1$ by assumption, we have
\[
i_{C*}\cG_2[-1] \in {}^\perp\Db_C\CohGGm(\cN_C)_{\le a} \cap \Db_C\CohGGm(\cN_C)_{\ge b}^\perp.  
\]
Thus,~\eqref{eqn:preglue-dt} shows that 
\[
\cF''|_V \cong \cF'|_V \cong \cF.
\qquad\text{and}\qquad
\cF'' \in {}^\perp\Db_C\CohGGm(\cN_C)_{\le a} \cap  \Db_C\CohGGm(\cN_C)_{\ge b}^\perp.
\]

We will now prove the stronger claim that 
\begin{equation}\label{eqn:preglue-induc}
\cF'' \in {}^\perp\Db_C\CohGGm(\cN_C)_{\le a+1} \cap  \Db_C\CohGGm(\cN_C)_{\ge b}^\perp.
\end{equation}
By Lemma~\ref{lem:pregluing}, it is enough to show that
\[
j_*\cF'' \in {}^\perp\Db\CohGGm_{\cN_C}(\fg_C)_{\le a+1}.
\]
Since we already know that $j_*\cF'' \in {}^\perp\Db\CohGGm_{\cN_C}(\fg_C)_{\le a}$, it is enough to show that
\begin{equation}\label{eqn:preglue-hom}
\Hom(j_*\cF'', j_{C*}\cT_\omega[n]\la k\ra)  = 0 \qquad\text{if $n+k = -a-1$.}
\end{equation}
The distinguished triangle $j_*\cF'' \to j_*\cF' \xrightarrow{\phi} j_{C*}\cG_2 \to$ gives rise to a long exact sequence
\begin{multline*}
\cdots \to
\Hom(j_{C*}\cG_2, j_{C*}\cT_\omega[n]\la k\ra) \to 
\Hom(j_*\cF', j_{C*}\cT_\omega[n]\la k\ra) \to \\
\Hom(j_*\cF'', j_{C*}\cT_\omega[n]\la k\ra) \to \Hom(j_{C*}\cG_2[-1], j_{C*}\cT_\omega[n]\la k\ra) \to \cdots
\end{multline*}
Here, the first map is surjective (see~\eqref{eqn:fpp-surj}), and the last term vanishes in view of the co-$t$-structure on $\Db_C\CohGGm(\fg_C)$.  We have now proved~\eqref{eqn:preglue-hom}, and hence~\eqref{eqn:preglue-induc}.

The construction carried out above shows how to modify the object $\cF'$ in such a way that the integer $a$ in~\eqref{eqn:preglue-fund} can be replaced by $a+1$.  The proof relies on the assumption that $a+1 < b-1$.  A similar (but ``dual'') construction lets us replace $b$ by $b-1$ (again assuming that $a+1 < b-1$).  Since we began with the assumption that $a \le -1$ and $b \ge 1$, these two constructions can be repeated until we arrive at an object $\tilde\cF$ as in the statement of the proposition.

It remains to prove the last assertion in the proposition.  For this, we return to allowing $U$ to be any $G$-stable open subset. Suppose $\cF$ is indecomposable.  The object $\tilde\cF$ obtained by the construction above is not necessarily indecomposable, but any direct summand of it still lies in ${}^\perp\Db_C\CohGGm(U)_{\le -1} \cap \Db_C\CohGGm(U)_{\ge 1}^\perp$, and it must have some indecomposable summand whose restriction to $U$ is isomorphic to $\cF$.
\end{proof}

\begin{lem}\label{lem:restrict-surj}
Let $U \subset \cN$ be a $G$-stable open subset that contains $C$ as a closed subset, and let $V = U \smallsetminus C$. If $\cF \in {}^\perp\Db_C\CohGGm(U)_{\le -1}$ and $\cG \in \Db_C\CohGGm(U)_{\ge 1}^\perp$, then the map
\[
\Hom(\cF,\cG[n]) \to \Hom(\cF|_V, \cG|_V[n])
\]
is surjective for $n = 0$, and an isomorphism for $n > 0$.
\end{lem}
\begin{proof}
Let $h: V \hookrightarrow U$ be the inclusion map.  In the derived category of quasicoherent sheaves $D^+\QCoh^{G \times \Gm}(U)$, we have a distinguished triangle
\[
R\Gamma_C(\cG) \to \cG \to h_*h^*\cG \to.
\]
From the long exact sequence obtained by applying $\Hom(\cF,{-})$ to this triangle, we see that it is enough to prove that
\[
\Hom(\cF, R\Gamma_C(\cG)[n]) = 0
\qquad\text{if $n > 0$.}
\]

Suppose we have a morphism $\phi: \cF \to R\Gamma_C(\cG)[n]$.  We will work with this map at the level of chain complexes, as follows: first replace $\cG$ by an injective resolution.  Then, since $\Gamma_C$ sends injective sheaves to injective sheaves, $R\Gamma_C(\cG)$ is a bounded-below complex of injective quasicoherent sheaves supported set-theoretically on $C$.  Since $\cF$ is a bounded complex of coherent sheaves, the image of the chain map $\phi: \cF \to R\Gamma_C(\cG)[n]$ is contained in some bounded subcomplex of coherent sheaves $\cE \subset R\Gamma_C(\cG)[n]$.  Of course, the terms of $\cE$ are also supported set-theoretically on $C$, so $\cE$ belongs to $\Db_C\CohGGm(U)$.  Using the co-$t$-structure on that category, we can find a distinguished triangle
\[
\cK_1 \to \cE \to \cK_2 \to
\]
with $\cK_1 \in \Db_C\CohGGm(U)_{\ge 0}$ and $\cK_2 \in \Db_C\CohGGm(U)_{\le -1}$.   Since $\cF$ belongs to ${}^\perp\Db_C\CohGGm(U)_{\le -1}$, we have $\Hom(\cF,\cK_2) = 0$, and any map $\cF \to  \cE$ factors through $\cK_1$.

To summarize, our map $\phi$ factors as a composition
\begin{equation}\label{eqn:restrict-surj-factor}
\cF \to \cK_1 \to \cE \to R\Gamma_C(\cG)[n].
\end{equation}
Since $\cK_1$ is set-theoretically supported on $C$, the map
\[
\Hom(\cK_1, R\Gamma_C(\cG)[n]) \simto \Hom(\cK_1[-n], \cG)
\]
is an isomorphism.  But since $\cG \in \Db_C\CohGGm(U)_{\ge 1}^\perp$, both of these $\Hom$-groups are $0$.  We conclude that the composition of maps in~\eqref{eqn:restrict-surj-factor} is zero, as desired.
\end{proof}

\begin{lem}\label{lem:silt-main}
Let $U \subset \cN$ be a $G$-stable open subset, and let $C \subset U$ be a nilpotent orbit.  For $\omega \in \Omega_C$, there is (up to isomorphism) a unique object $\cS_U(C,\cT_\omega) \in \Db\CohGGm(U)$ with the following properties:
\begin{enumerate}
\item $\cS_U(C,\cT_\omega)$ is indecomposable.\label{it:s-indecomp}
\item $\cS_U(C,\cT_\omega)$ is supported set-theoretically on $\overline{C} \cap U$.\label{it:s-support}
\item If $i_C: C \hookrightarrow \cN_C \cap U$ denotes the inclusion map, then\label{it:s-tilt}
\[
\textstyle
\cS_U(C,\cT_\omega)|_{\cN_C \cap U} \cong i_{C*}\cT_\omega[-\frac{1}{2}\codim C]\la \frac{1}{2}\codim C\ra.
\]
\item For each orbit $C' \subset \partial C \cap U$, we have\label{it:s-pullback}
\begin{align*}
j_{C'}^!j_*\cS_U(C,\cT_\omega) &\in \Db\CohGGm(C')_{\le 0}, \\
j_{C'}^*j_*\cS_U(C,\cT_\omega) &\in \Db\CohGGm(C')_{\ge 0}.
\end{align*}
\end{enumerate}
Moreover, this object has the following additional properties:
\begin{enumerate}
\setcounter{enumi}{4}
\item If $V \subset U$ is a smaller $G$-stable open subset that contains $C$, then\label{it:s-open}
\[
\cS_U(C,\cT_\omega)|_V \cong \cS_V(C,\cT_\omega).
\]
\item If $C_0$ is a nilpotent orbit that is closed in $U$, and if $C \ne C_0$, then\label{it:s-perp}
\[
\cS_U(C,\cT_\omega) \in {}^\perp\Db_{C_0}\CohGGm(U)_{\le -1} \cap \Db_{C_0}\CohGGm(U)_{\ge 1}^\perp.
\]
\end{enumerate}
\end{lem}
\begin{proof}
We proceed by induction on the number of orbits in $U$.  If $U$ is empty, there is nothing to prove.  Otherwise, assume that the lemma is already known to hold when $U$ is replaced by any smaller open subset.  

We will prove the existence of uniqueness of $\cS_U(C,\cT_\omega)$ satisfying~\eqref{it:s-indecomp}--\eqref{it:s-pullback}.  Before doing this, let us show if $\cS_U(C,\cT_\omega)$ exists, it automatically satisfies properties~\eqref{it:s-open} and~\eqref{it:s-perp} as well.  Indeed, property~\eqref{it:s-perp} follows from property~\eqref{it:s-pullback} by Lemma~\ref{lem:pregluing} and~\eqref{eqn:pregluing-crit}.  For property~\eqref{it:s-open}, by induction, it is enough prove it in the special case where $V$ is the complement of a single closed $G$-orbit in $U$.  Let $C_0$ be such a closed orbit, and let $V = U \smallsetminus C_0$.  Of course, property~\eqref{it:s-open} is vacuous if $C = C_0$.  If $C \ne C_0$, then by Lemma~\ref{lem:restrict-surj}, restriction to $V$ induces a surjective ring homomorphism
\[
\End(\cS_U(C,\cT_\omega)) \twoheadrightarrow \End(\cS_U(C,\cT_\omega)|_V).
\]
Since $\cS_U(C,\cT_\omega)$ is indecomposable, the left-hand side above is a local ring, and so the right-hand side is as well.  We have shown that $\cS_U(C,\cT_\omega)|_V$ is indecomposable.  It is easy to see that $\cS_U(C,\cT_\omega)|_V$ satisfies properties~\eqref{it:s-support}--\eqref{it:s-pullback} on $V$.  Property~\eqref{it:s-open} then follows by the uniqueness of $\cS_V(C,\cT_\omega)$ (which holds by induction).

Let us now prove the existence of $\cS_U(C,\cT_\omega)$ satisfying~\eqref{it:s-indecomp}--\eqref{it:s-pullback}.  Choose a closed $G$-orbit $C_0 \subset U$, and let $V = U \smallsetminus C_0$ as above.  If $C = C_0$, it is obvious that
\begin{equation}\label{eqn:silt-closed-defn}
\textstyle
\cS_U(C_0, \cT_\omega) = i_{C_0*}\cT_\omega[-\frac{1}{2}\codim C_0]\la \frac{1}{2}\codim C_0\ra
\end{equation}
has the required properties.  Otherwise, for $C \ne C_0$, we define
\[
\cS_U(C,\cT_\omega) =
\begin{array}{c}
\text{an indecomposable extension of $\cS_{V}(C,\cT_\omega)$ that} \\
\text{satisfies~\eqref{it:s-perp}, obtained from Proposition~\ref{prop:preglue-extend}.}
\end{array}
\]
This object obviously satisfies properties~\eqref{it:s-indecomp}--\eqref{it:s-tilt}, as well as property~\eqref{it:s-pullback} for $C' \ne C_0$.  Property~\eqref{it:s-pullback} for $C' = C_0$ follows from property~\eqref{it:s-perp} via Lemma~\ref{lem:pregluing} and~\eqref{eqn:pregluing-crit}.  This completes the proof of existence.

Finally, we must prove uniqueness. For $C = C_0$, this is obvious.  Suppose now that $C \ne C_0$.  If we had two objects $\cS_U^{(1)}(C,\cT_\omega)$ and $\cS_U^{(2)}(C,\cT_\omega)$ both satisfying~\eqref{it:s-indecomp}--\eqref{it:s-pullback}, then they would both satisfy properties~\eqref{it:s-open} and~\eqref{it:s-perp} as well.  By Lemma~\ref{lem:restrict-surj}, restriction to $V$ induces a surjective map
\[
\Hom(\cS_U^{(1)}(C,\cT_\omega), \cS_U^{(2)}(C,\cT_\omega)) \twoheadrightarrow \End(\cS_{V}(C,\cT_\omega)).
\]
In particular, there exists a map $\phi_1: \cS_U^{(1)}(C,\cT_\omega) \to \cS_U^{(2)}(C,\cT_\omega)$ such that $\phi_1|_{V} = \id_{\cS_{V}(C,\cT_\omega)}$.  Similarly, there exists a map $\phi_2: \cS_U^{(2)}(C,\cT_\omega) \to \cS_U^{(1)}(C,\cT_\omega)$ that also satisfies $\phi_2|_{V} = \id$.  The composition $\phi_2 \phi_1$ is an endomorphism of $\cS_U^{(1)}(C,\cT_\omega) $ whose image under the surjective ring homomorphism
\[
\End(\cS_U^{(1)}(C,\cT_\omega)) \twoheadrightarrow \End(\cS_{V}(C,\cT_\omega))
\]
is the identity map.  Because this is a homomorphism of local rings, we deduce that $\phi_2 \phi_1$ is invertible.  Likewise, $\phi_1 \phi_2$ is invertible.  It follows that $\phi_1$ and $\phi_2$ are themselves isomorphisms.  This completes the proof of uniqueness.
\end{proof}

\begin{thm}\label{thm:orbitwise}
Let $U \subset \cN$ be a $G$-stable open subset.  The full additive subcategory of $\Db\CohGGm(U)$ consisting of direct sums of objects of the form
\[
\cS_U(C,\cT_\omega)[n]\la  -n\ra
\qquad\text{with $C \subset U$, $\omega \in \Omega_C$, and $n \in \Z$}
\]
is a silting subcategory.
\end{thm}
\begin{proof}
It is easy to see that the collection of objects $\{ \cS_U(C,\cT_\omega)[n]\la -n\ra\}$ generates $\Db\CohGGm(U)$ as a triangulated category.  Thus, as in Propositions~\ref{prop:lie-extend} and~\ref{prop:nil-extend}, the result comes down to showing that
\begin{equation}\label{eqn:silt-ext}
\Hom(\cS_U(C,\cT_\omega), \cS_U(C',\cT_{\omega'})[n]\la k\ra) = 0
\qquad\text{if $n + k > 0$.}
\end{equation}
To prove this, we proceed by induction on the number of orbits in $U$.  If $U$ is empty, there is nothing to prove.  Otherwise, let $C_0$ be an orbit that is closed in $U$, and let $U' = U \smallsetminus C_0$.
We consider various cases as follows:
\begin{itemize}
\item If $C = C' = C_0$, then, in view of~\eqref{eqn:silt-closed-defn}, the claim~\eqref{eqn:silt-ext} is just part of Proposition~\ref{prop:nil-extend}.
\item If $C = C_0$ but $C' \subset U'$, then~\eqref{eqn:silt-ext} holds because $\cS_U(C',\cT_{\omega'})$ satisfies property~\eqref{it:s-perp} from Lemma~\ref{lem:silt-main}.
\item If $C \subset U'$ but $C' = C_0$, then~\eqref{eqn:silt-ext} holds because $\cS_U(C,\cT_{\omega})$ satisfies property~\eqref{it:s-perp} from Lemma~\ref{lem:silt-main}.
\item Finally, if both $C$ and $C'$ are contained in $U'$, then Lemma~\ref{lem:silt-main}\eqref{it:s-perp} and Lemma~\ref{lem:restrict-surj} together tell us that restriction to $U'$ gives an isomorphism
\[
\Hom(\cS_U(C,\cT_\omega), \cS_U(C',\cT_{\omega'})[n]\la k\ra)
\simto
\Hom(\cS_{U'}(C,\cT_\omega), \cS_{U'}(C',\cT_{\omega'})[n]\la k\ra)
\]
whenever $n +k > 0$.  The right-hand side vanishes by induction.\qedhere
\end{itemize}
\end{proof}

The silting subcategory described by the preceding theorem is denoted by
\[
\Silt(U) \subset \Db\CohGGm(U).
\]
Note that for any $G$-stable closed subset $Z \subset U$, the collection $\{ \cS_U(C,\cT_\omega)[n]\la -n\ra \mid C \subset Z \}$ generates $\Db_Z\CohGGm(U)$ as a triangulated category.  It follows immediately that the category
\[
\Silt_Z(U) = \Silt(U) \cap \Db_Z\CohGGm(U)
\]
is a silting subcategory of $\Db_Z\CohGGm(U)$.

\begin{lem}\label{lem:silt-open}
Let $V \subset U \subset \cN$ be two $G$-stable open subsets, and let $Z \subset U$ be a $G$-stable subset that is closed in $U$.  The restriction functor
\[
\Db_Z\CohGGm(U) \to \Db_{Z \cap V}\CohGGm(V)
\qquad\text{given by}\qquad
\cF \mapsto \cF|_V
\]
is co-$t$-exact.  Moreover, if $\cF \in \Db_Z\CohGGm(U)_{\ge 0}$ and $\cG \in \Db_Z\CohGGm(U)_{\le 0}$, then the map
\begin{equation}\label{eqn:silt-open-surj}
\Hom(\cF,\cG) \to \Hom(\cF|_V, \cG|_V)
\end{equation}
is surjective.
\end{lem}
\begin{proof}
The co-$t$-exactness of the restriction functor is an immediate consequence of Lemma~\ref{lem:silt-main}\eqref{it:s-open}.  For the surjectivity of~\eqref{eqn:silt-open-surj}, by induction on the number of orbits in $U \smallsetminus V$, we can reduce to the case where $V$ is the complement of a single closed orbit $C_0$.  In this case, by Lemma~\ref{lem:silt-main}\eqref{it:s-perp}, the assumption that $\cF \in \Db_Z\CohGGm(U)_{\ge 0}$ implies that $\cF$ lies in ${}^\perp\Db_{C_0}\CohGGm(U)_{\le -1}$.  Similarly, $\cG$ lies in $\Db_{C_0}\CohGGm(U)_{\ge 1}^\perp$.  By Lemma~\ref{lem:restrict-surj}, we are done.
\end{proof}

\begin{lem}\label{lem:silt-closed-cond}
Let $U \subset \cN$ be a $G$-stable open subset, and let $Z \subset U$ be a $G$-stable subset that is closed in $U$.  For $\cF \in \Db\CohGGm(U)$, the following three conditions are equivalent:
\begin{enumerate}
\item For each orbit $C' \subset Z$, we have $j_{C'}^!j_*\cF \in \Db\CohGGm(C')_{\le 0}$.\label{it:scc-pull}
\item We have $\Hom(\cG,\cF) = 0$ for all $\cG \in \Db_Z\CohGGm(U)_{\ge 1}$.\label{it:scc-hom}
\item We have $\Hom(\cS_U(C',\cT_\omega)[n]\la k\ra, \cF) = 0$ for $C' \subset Z$, $\omega \in \Omega_C$, and $n +k  \le -1$.\label{it:scc-obj}
\end{enumerate}
Similarly, the following three conditions are equivalent:
\begin{enumerate}
\item For each orbit $C' \subset Z$, we have $j_{C'}^*j_*\cF \in \Db\CohGGm(C')_{\ge 0}$.
\item We have $\Hom(\cF,\cG) = 0$ for all $\cG \in \Db_Z\CohGGm(U)_{\le -1} $.
\item We have $\Hom(\cF,\cS_U(C',\cT_\omega)[n]\la k\ra) = 0$ for $C' \subset Z$, $\omega \in \Omega_C$, and $n+k \ge 1$.
\end{enumerate}
\end{lem}
\begin{proof}
We will prove the equivalence of the first set of three statements.  The proof for the second set of three statements is similar.  In the induced co-$t$-structure on $\Db_Z\CohGGm(U)$, the category $\Db_Z\CohGGm(U)_{\ge 1}$ is generated under extensions by objects of the form $\cS_U(C',\cT_\omega)[n]\la k\ra$ with $C' \subset Z$ and $n+k \le -1$.  This observation yields the equivalence of statements~\eqref{it:scc-hom} and~\eqref{it:scc-obj}.

To prove the equivalence of~\eqref{it:scc-pull} with the other two conditions, we proceed by induction on the number of orbits in $Z$.  If $Z$ consists of a single closed orbit $C_0$, then condition~\eqref{it:scc-hom} just says that $\cF \in \Db_{C_0}\CohGGm(U)_{\ge 1}^\perp$.  Using Lemmas~\ref{lem:perp-crit} and~\ref{lem:pregluing} together with~\eqref{eqn:pregluing-crit}, we see that this condition is equivalent to~\eqref{it:scc-pull}.

If $Z$ contains more than one orbit, choose a closed orbit $C_0 \subset Z$, and let $V = U \smallsetminus C_0$.  Then, by induction, condition~\eqref{it:scc-pull} is equivalent to
\begin{enumerate}
\setcounter{enumi}{3}
\item We have $\cF \in \Db_{C_0}\CohGGm(U)_{\ge 1}^\perp$, and $\Hom(\cS_V(C',\cT_\omega)[n]\la k\ra, \cF|_V) = 0$ for $C' \subset Z \cap V$, $\omega \in \Omega_C$, and $n +k  \le -1$.\label{it:scc-ind}
\end{enumerate}
By Lemma~\ref{lem:restrict-surj}, the map $\Hom(\cS_U(C',\cT_\omega)[n]\la k\ra, \cF) \to \Hom(\cS_V(C',\cT_\omega)[n]\la k\ra, \cF|_V)$ is an isomorphism when $n+k \le -1$.  It follows that condition~\eqref{it:scc-ind} is equivalent to~\eqref{it:scc-obj}.
\end{proof}

\begin{lem}
\phantomsection\label{lem:total-duality}
\begin{enumerate}
\item For any $G$-stable open subset $U \subset \cN$, the category $\Silt(U) \subset \Db\CohGGm(U)$ is preserved by the Serre--Grothendieck duality functor.  Specifically, we have\label{it:total-sg}
\[
\D(\cS_U(C,\cT_\omega)[n]\la -n\ra) \cong \cS_U(C,\cT_\omega^*)[-n]\la n\ra.
\]
\item The category $\Silt(\cN) \subset \Db\CohGGm(\cN)$ is preserved by the opposition functor $({-})^\sigma: \Db\CohGGm(\cN) \to \Db\CohGGm(\cN)$.  Specifically, we have\label{it:total-opp}
\[
(\cS(C,\cT_\omega)[n]\la -n\ra)^\sigma \cong \cS(C,\cT_\omega^*)[n]\la -n\ra.
\]
\end{enumerate}
\end{lem}
\begin{proof}
It is enough to prove these statements in the special case $n = 0$.  

\eqref{it:total-sg}~It is clear that $\D\cS_U(C,\cT_\omega)$ is indecomposable and supported on $\overline{C} \cap U$.  By Lemma~\ref{lem:nilp-duality}, we have
\begin{align*}
(\D\cS_U(C,\cT_\omega))|_{\cN_C \cap U}
&\textstyle\cong \D(i_{C*}\cT_\omega[-\frac{1}{2}\codim C]\la \frac{1}{2}\codim C\ra)  \\
&\textstyle\cong i_{C*}\cT_\omega^*[-\frac{1}{2}\codim C]\la \frac{1}{2}\codim C\ra.
\end{align*}
Finally, if $C' \subset \partial C \cap U$, then by~\eqref{eqn:serre-groth-compat}, \eqref{eqn:serre-groth-open}, and Lemma~\ref{lem:repsilt-duality}, we find that
\begin{align*}
j_{C'}^!j_*\D(\cS_U(C,\cT_\omega)) &\in \Db\CohGGm(C')_{\le 0}, \\
j_{C'}^*j_*\D(\cS_U(C,\cT_\omega)) &\in \Db\CohGGm(C')_{\ge 0}.
\end{align*}
We have shown that $\D\cS_U(C,\cT_\omega)$ satisfies the conditions from Lemma~\ref{lem:silt-main} that uniquely characterize $\cS_U(C,\cT_\omega^*)$.  Part~\eqref{it:total-sg} of the lemma follows.

\eqref{it:total-opp}~It is clear that $\cS(C,\cT_\omega)^\sigma$ is indecomposable, and~\cite[Corollary~4.2]{ah:pgl3} implies that it is supported on $\overline{C}$.  The claim that
\[
\textstyle
(\cS(C,\cT_\omega))|_{\cN_C} \cong i_{C*}\cT_\omega^*[-\frac{1}{2}\codim C]\la \frac{1}{2}\codim C\ra
\]
follows from~\eqref{eqn:sigma-orbit} applied to the coherent sheaf $\cH^{\frac{1}{2}\codim C}(\cS(C,\cT_\omega))$, along with Lemma~\ref{lem:tilting-oppo}.  We have shown that $\cS(C,\cT_\omega)^\sigma$ satisfies the first three conditions from Lemma~\ref{lem:silt-main} characterizing $\cS(C,\cT_\omega^*)$.

We will now prove that $\cS(C,\cT_\omega)^\sigma \cong \cS(C,\cT_\omega^*)$ by induction on $C$ with respect to the closure partial order on nilpotent orbits.  In view of the preceding paragraph, it is enough to check condition~\eqref{it:s-pullback} from Lemma~\ref{lem:silt-main} for $C' \subset \partial C$.  If $C$ is the zero nilpotent orbit, then  $\partial C$ is empty, and there is nothing to prove.  Now suppose $C$ is not the zero orbit.  By induction and the fact that $({-})^\sigma$ is an equivalence of categories, we have
\begin{align*}
\Hom(\cS(C,\cT_\omega)^\sigma, \cS(C',\cT_{\omega'})^\sigma[n]\la k\ra)
&\cong \Hom(\cS(C,\cT_\omega), \cS(C',\cT_{\omega'})[n]\la k\ra) \\
&\cong \Hom(\cS(C,\cT_\omega)^\sigma, \cS(C',\cT_{\omega'}^*)[n]\la k\ra).
\end{align*}
If $n + k \ge 1$, the second expression above vanishes.  As $\omega$ varies over $X_{C'}^+$, the object $\cS(C',\cT_{\omega'}^*)$ varies over all silting objects whose support is $\overline{C'}$.  By Lemma~\ref{lem:silt-closed-cond}, the vanishing of the third expression above implies that
\[
j_{C'}^*j_*(\cS(C,\cT_\omega)^\sigma) \in \Db\CohGGm(C')_{\ge 0}
\qquad\text{for all $C' \subset \partial C$}.
\]
A similar argument shows that 
\[
j_{C'}^!j_*(\cS(C,\cT_\omega)^\sigma) \in \Db\CohGGm(C')_{\le 0}
\qquad\text{for all $C' \subset \partial C$},
\]
and thus $\cS(C,\cT_\omega)^\sigma$ satisfies condition~\eqref{it:s-pullback} from Lemma~\ref{lem:silt-main}.
\end{proof}

%%%%%%%%%%%%%%%%%%%%%%%%%%%%%%%%%%%%%%%%%%%%%%%%%%%%%%%%%%%%%%%%%%%%%%%%%%%
\section{Silting objects and the Lusztig--Vogan bijection}
\label{sec:lv}
%%%%%%%%%%%%%%%%%%%%%%%%%%%%%%%%%%%%%%%%%%%%%%%%%%%%%%%%%%%%%%%%%%%%%%%%%%%

In Theorem~\ref{thm:orbitwise} above, we have defined a co-$t$-structure on $\Db\CohGGm(\cN)$, which we will call the \emph{orbitwise co-$t$-structure}.  On the other hand, in~\cite{ah:cot1}, the authors defined another co-$t$-structure on $\Db\CohGGm(\cN)$, called the \emph{supportive co-$t$-structure}.  In this section, we will briefly review the definition of the supportive co-$t$-structure, and then we will prove that these two co-$t$-structures coincide.

%--------------------------------------------------------------------------
\subsection{Review of the Lusztig--Vogan bijection}
%--------------------------------------------------------------------------

Let $\PCoh^{\Gm}(\cN)$ be the category of $G \times \Gm$-equivariant perverse-coherent sheaves on $\cN$ (see~\cite{a:pcsnc, bez:qes}).  Recall that this is the heart of a $t$-structure on $\Db\CohGGm(\cN)$.  Furthermore, every object in this abelian category has finite length.  The simple objects can be described in terms of irreducible vector bundles on nilpotent orbits.  Specifically, for $\omega \in \Omega_C$, let $\cL_\omega \in \CohGGm(C)$ be the corresponding irreducible vector bundle.  Then, for each nilpotent orbit $C \subset \cN$ and each $\omega \in \Omega_C$, there is a unique simple object
\[
\IC(C,\cL_\omega) \in \PCoh^{\Gm}(\cN)
\]
that is supported on $\overline{C}$ and satisfies
\[
\textstyle
\IC(C,\cL_\omega)|_{\cN_C} \cong i_{C*}\cL_\omega[-\frac{1}{2}\codim C]\la \frac{1}{2}\codim C\ra.
\]
Moreover, every simple object is isomorphic to some $\IC(C,\cL_\omega)\la n\ra$.

This $t$-structure also admits a description in terms of push-forwards of line bundles along the Springer resoluion $\pi: \tcN \to \cN$, where $\tcN = T^*(G/B)$ is the cotangent bundle of the flag variety of $G$.  Let $\bX$ be the weight lattice of $G$, and let $\bX^+ \subset \bX$ be the set of dominant weights.  Any weight $\lambda \in \bX$ determines a line bundle $\cO_{\tcN}(\lambda) \in \Db\CohGGm(\tcN)$.  Now assume that $\lambda \in \bX^+$.  Let $\delta_\lambda^*$ denote the length of the shortest element $w \in W$ such that $w\lambda \in -\bX^+$, and then set
\[
\oDelta_\lambda = \pi_*\cO_{\tcN}(w_0\lambda)\la \delta_\lambda^*\ra,
\qquad
\onabla_\lambda = \pi_*\cO_{\tcN}(\lambda)\la -\delta_\lambda^*\ra.
\]
According to~\cite{a:pcsnc, bez:qes}, these objects lie in $\PCoh^{\Gm}(\cN)$.  Moreover, for each $\lambda \in \bX^+$, there is a canonical map
\begin{equation}\label{eqn:iota}
\iota_\lambda: \oDelta_\lambda \to \onabla_\lambda
\end{equation}
whose image (in the abelian category $\PCoh^{\Gm}(\cN)$) is a simple object.  That is, there is a unique pair $(C,\omega) \in \Omega$ such that $\iota_\lambda$ factors as
\[
\oDelta_\lambda \twoheadrightarrow \IC(C,\cL_\omega) \hookrightarrow \onabla_\lambda.
\]
(A priori, the image might have been of the form $\IC(C,\cL_\omega)\la n\ra$ for some $n \in \Z$; this integer was determined to be $0$ in~\cite{ah:pgl3}.)  The resulting map
\[
\Theta_\LV: \bX^+ \simto \Omega
\]
is in fact a bijection, known as the \emph{Lusztig--Vogan bijection}.

The usual partial order $\le$ on $\bX^+$ (given by declaring that $\lambda \le \mu$ if $\mu - \lambda$ is a sum of positive roots) is related to the Lusztig--Vogan bijection as follows: if $\IC(C',\cL_{\omega'})\la n\ra$ occurs as a composition factor of $\oDelta_\lambda$ (or $\onabla_\lambda$), then the pair $(C',\omega')$ corresponds under the Lusztig--Vogan bijection to a weight $\mu \le \lambda$.

Now let
\begin{align*}
D(\le \lambda) &= 
\begin{array}{c}
\text{the full triangulated subcategory of $\Db\CohGGm(\cN)$} \\
\text{generated by $\oDelta_\mu\la n\ra$ with $\mu \le \lambda$ and $n \in \Z$.}
\end{array}
\\
&=
\begin{array}{c}
\text{the full triangulated subcategory of $\Db\CohGGm(\cN)$} \\
\text{generated by $\onabla_\mu\la n\ra$ with $\mu \le \lambda$ and $n \in \Z$.}
\end{array}
\end{align*}
(The fact that this category can be defined either in terms of the $\oDelta_\mu$'s or the $\onabla_\mu$'s comes from the theory of quasi-exceptional sets, developed in~\cite{bez:qes, a:pcsnc, ah:cot1}.)  The preceding paragraph can be reformulated as follows: for $\cF \in \Db\CohGGm(\cN)$, we have
\begin{equation}\label{eqn:pcoh-serre}
\cF \in D(\le \lambda)
\Longleftrightarrow
\begin{array}{c}
\text{every composition factor $\IC(C,\cL_\omega)\la n\ra$ of every} \\
\text{${}^p\cH^i(\cF)$ corresponds to a weight $\mu \le \lambda$ under $\Theta_\LV$}
\end{array}
\end{equation}
Here, ${}^p\cH^i({-})$ denotes cohomology with respect to the perverse-coherent $t$-structure.

%--------------------------------------------------------------------------
\subsection{The supportive co-\texorpdfstring{$t$}{t}-structure}
%--------------------------------------------------------------------------

According to~\cite[Proposition~4.3]{ah:cot1} and the remarks following it, the objects $\oDelta_\lambda$ and $\onabla_\lambda$ also determine a co-$t$-structure on $\Db\CohGGm(\cN)$, known as the \emph{supportive co-$t$-structure}, and given by
\begin{align*}
\Db\CohGGm(\cN)^{\supp}_{\ge 0} 
&= 
\begin{array}{@{}c@{}}
\text{the full subcategory generated under extensions and direct} \\
\text{summands by the $\oDelta_\lambda[n]\la k\ra$ for $\lambda \in \bX^+$ and $n+k \le 0$}
\end{array}
\\
\Db\CohGGm(\cN)^{\supp}_{\le 0} 
&= 
\begin{array}{@{}c@{}}
\text{the full subcategory generated under extensions and direct} \\
\text{summands by the $\onabla_\lambda[n]\la k\ra$ for $\lambda \in \bX^+$ and $n+k \ge 0$}
\end{array}
\end{align*}
Moreover,~\cite[Proposition~2.22]{ah:cot1} gives a classification of the indecomposable silting objects for this category: for each $\lambda \in \bX^+$, there is a unique (up to isomorphism) indecomposable silting object
\[
\fS_\lambda \in \Db\CohGGm(\cN)^{\supp}_{\ge 0} \cap \Db\CohGGm(\cN)^{\supp}_{\le 0}
\]
that is characterized by the following two properties: the object $\fS_\lambda$ lies in $D(\le \lambda)$, and the canonical map~\eqref{eqn:iota} factors as
\[
\oDelta_\lambda \to \fS_\lambda \to \onabla_\lambda.
\]
Every indecomposable silting object is isomorphic to $\fS_\lambda[n]\la -n\ra$ for some $\lambda \in \bX^+$ and $n \in \Z$.

\begin{lem}
\phantomsection\label{lem:supp-duality}
\begin{enumerate}
\item The supportive co-$t$-structure is preserved by the Serre--Gro\-then\-dieck duality functor.  Specifically, we have
\[
\D(\fS_\lambda[n]\la -n\ra) \cong \fS_{-w_0\lambda}[-n]\la n\ra.
\]
\item The supportive co-$t$-structure is preserved by the opposition functor
$({-})^\sigma: \Db\CohGGm(\cN) \to \Db\CohGGm(\cN)$.  Specifically, we have
\[
(\fS_\lambda[n]\la -n\ra)^\sigma \cong \fS_{-w_0\lambda}[n]\la -n\ra.
\]
\end{enumerate}
\end{lem}
\begin{proof}
According to~\cite[Eq.~(3.6) and Proposition~4.1]{ah:pgl3}, we have
\begin{align*}
\D(\oDelta_\lambda) &\cong \onabla_{-w_0\lambda},
&
(\oDelta_\lambda)^\sigma &\cong \oDelta_{-w_0\lambda}, \\
\D(\onabla_\lambda) &\cong \oDelta_{-w_0\lambda},
&
(\onabla_\lambda)^\sigma &\cong \onabla_{-w_0\lambda}.
\end{align*}
The lemma follows immediately.  (In the case of $\D$, these isomorphisms have been known for much longer: see~\cite{a:pcsnc, bez:qes}.)
\end{proof}

\begin{thm}\label{thm:co-t-struc}
The orbitwise and supportive co-$t$-structures coincide.  More precisely, we have
\[
\fS_\lambda \cong \cS(C,\cT_\omega)
\]
where $\lambda$ corresponds to $(C,\omega)$ under the Lusztig--Vogan bijection.
\end{thm}
\begin{proof}
Both the orbitwise and supportive co-$t$-structures are preserved by the functor $[1]\la -1\ra$.  Next, let $\delta: \Db\CohGGm(\cN)^\op \to \Db\CohGGm(\cN)$ be the functor given by $\delta(\cF) = \D(\cF)^\sigma$.  By Lemmas~\ref{lem:total-duality} and~\ref{lem:supp-duality}, both co-$t$-structures are preserved by $\delta$; moreover, for each indecomposable silting object $\cF$ (with respect to either co-$t$-structure), there is a unique integer $n$ such that $\delta(\cF[n]\la -n\ra) \cong \cF$.  Therefore, by Lemma~\ref{lem:duality-unique}, the two co-$t$-structures coincide.

In particular, for each $\lambda \in \bX^+$,  the object $\fS_\lambda$ is also an indecomposable silting object in the orbitwise co-$t$-structure, so it is isomorphic to $\cS(C,\cT_\omega)[n]\la -n\ra$ for some $(C,\omega) \in \Omega$ and some $n \in \Z$.  Since $\fS_\lambda$ is preserved by $\delta$, we must in fact have $n = 0$.  The isomorphisms $\fS_\lambda \cong \cS(C,\cT_\omega)$ determine a bijection
\[
\Theta: \bX^+ \simto \Omega.
\]
To finish the proof, it remains to show that $\Theta = \Theta_\LV$.

Suppose $\Theta(\lambda) = (C,\omega)$.  The description of $\cS(C,\cT_\omega)$ shows that $\cS(C,\cT_\omega)|_{\cN_C}$ is a perverse-coherent sheaf, and that $i_{C*}\cL_\omega[-\frac{1}{2}\codim C]\la \frac{1}{2}\codim C\ra$ occurs as a composition factor therein.  The restriction from $\cN$ to $\cN_C$ is $t$-exact for the perverse-coherent $t$-structure, and it sends any simple object to either $0$ or a simple object.  It follows that $\IC(C,\cL_\omega)$ occurs as a composition factor of ${}^p\cH^0(\cS(C,\cT_\omega))$.  Now let $\lambda' = \Theta_\LV^{-1}(C,\omega)$. Since $\cS(C,\cT_\omega) = \fS_\lambda$ lives in $D(\le \lambda)$, we see from~\eqref{eqn:pcoh-serre} that $\lambda' \le \lambda$, or in other words, that
\begin{equation}\label{eqn:theta-ineq}
\Theta_\LV^{-1}(\Theta(\lambda)) \le \lambda.
\end{equation}
This holds for all $\lambda \in \bX^+$.  In particular, if $\lambda$ is minimal with respect to $\le$, then $\Theta_\LV^{-1}(\Theta(\lambda)) = \lambda$.  For any $\lambda \in \bX^+$, the set $\{ \mu \in \bX^+ \mid \mu \le \lambda \}$ is finite, so by induction with respect to the partial order $\le$, we see that~\eqref{eqn:theta-ineq} implies that $\Theta_\LV^{-1}(\Theta(\lambda)) = \lambda$ for all $\lambda$, as desired.
\end{proof}

When $\bk$ has characteristic $0$, the representation theory of $G^{x_C}_\red$ is semisimple.  In particular, each indecomposable tilting vector bundle $\cT_\omega$ coincides with the irreducible vector bundle $\cL_\omega$.  In this situation, we obtain the following alternative description of $\cS(C,\cL_\omega)$.

\begin{cor}\label{cor:silt-char0}
If $\bk$ has characteristic $0$, then for every nilpotent orbit $C \subset \cN$ and every $\omega \in \Omega_C$, we have $\cS(C,\cL_\omega) \cong \IC(C,\cL_\omega)$.
\end{cor}
The following proof is short, but requires some notions that are not used elsewhere in this paper.  We refer the reader to~\cite{a:nepcs, ah:cot1, ahr:hcsv} for the relevant background.
\begin{proof}[Proof sketch]
Let $\lambda = \Theta_\LV^{-1}(C,\omega)$, so that $\cS(C,\cL_\omega) \cong \fS_\lambda$. According to~\cite[Theorem~4.5]{ah:cot1}, $\fS_\lambda$ is isomorphic to $\pi_*\widetilde{\fS}_{w_0\lambda}$, where $\pi: \tcN \to \cN$ is the Springer resolution, and $\widetilde{\fS}_{w_0\lambda}$ is a certain silting object on $\tcN$.  By~\cite[Theorem~3.9]{ahr:hcsv} (along with~\cite[Lemma~3.4]{ah:cot1}), $\widetilde{\fS}_{w_0\lambda}$ coincides with the simple exotic sheaf $\mathfrak{L}_{w_0\lambda}$.  Finally, by~\cite[Proposition~2.6]{a:nepcs}, $\pi_*\mathfrak{L}_{w_0\lambda}$ is the simple perverse-coherent sheaf corresponding to $\lambda$ under the Lusztig--Vogan bijection, i.e., $\IC(C,\cL_\omega)$.
\end{proof}

\begin{rmk}
The $\IC$ construction is defined for arbitrary equivariant vector bundles on an orbit, not just irreducible ones, so one may ask whether Corollary~\ref{cor:silt-char0} generalizes to positive characteristic with a statement of the form
\begin{equation}\label{eqn:silt-ic}
\cS(C,\cT_\omega) \cong \IC(C,\cT_\omega).
\end{equation}
This turns out to be false: the calculations in~\cite{ah:pgl3} yield counterexamples when $G = \mathrm{PGL}_3$, and $C$ is the subregular nilpotent orbit.

However,~\eqref{eqn:silt-ic} is (trivially) true when $C$ is the zero nilpotent orbit.  It is also true when $C$ is the \emph{regular} nilpotent orbit.  In this case, by~\cite[Proposition~4.8]{a:nepcs}, $(C,\omega)$ corresponds to a minuscule weight $\lambda \in \bX^+$.  The weight $w_0\lambda \in \bX$, which is ``antiminuscule'' in the terminology of~\cite{a:nepcs}, is minimal with respect to the partial order on $\bX$ used to define the both the exotic $t$-structure and supportive co-$t$-structure.  One can then deduce from the construction in~\cite{ah:cot1} that $\widetilde{\fS}_{w_0\lambda} \cong \mathfrak{L}_{w_0\lambda}$, and then the claim follows by the reasoning in the proof of Corollary~\ref{cor:silt-char0}.
\end{rmk}

%%%%%%%%%%%%%%%%%%%%%%%%%%%%%%%%%%%%%%%%%%%%%%%%%%%%%%%%%%%%%%%%%%%%%%%%%%%
\section{Proof of the relative Humphreys conjecture}
\label{sec:humphreys}
%%%%%%%%%%%%%%%%%%%%%%%%%%%%%%%%%%%%%%%%%%%%%%%%%%%%%%%%%%%%%%%%%%%%%%%%%%%

In this section, we assume that the characteristic $p$ of $\bk$ is larger than the Coxeter number $h$ for $G$.  Let $\mbf{G}$ be a reductive group whose first Frobenius twist $\mbf{G}^{(1)}$ is identified with $G$.  (Of course, this implies that $G$ and $\mbf{G}$ are isomorphic, but we do not identify them, as they play different roles in the discussion below.)  Let $\mbf{G}_{1}$ be the first Frobenius kernel of $\mbf{G}$.  The \emph{$\mbf{G}_{1}$-cohomology} of a $\mbf{G}$-module $M$ is defined by
\[
H^\bullet(\mbf{G}_{1},M) = \Ext^\bullet_{\mbf{G}_{1}}(\bk,M).
\]
A classical result~\cite{aj, fp:claag} (see also~\cite[Lemma~8.1]{ahr:hcsv}) states that there is a $G$-equivariant isomorphism of graded rings
\[
\Ext^\bullet_{\mbf{G}_{1}}(\bk,\bk) \cong \bk[\cN].
\]
Thus, for any $\mbf{G}$-module $M$, its $\mbf{G}_{1}$-cohomology has the structure of a graded $G$-equivariant module over $\bk[\cN]$, or equivalently, an object of $\CohGGm(\cN)$.  The goal of this section is to describe this module in the case where $M$ is a tilting $\mbf{G}$-module.  For $\lambda \in \bX^+$, let $\tilt(\lambda)$ be the indecomposable tilting $\mbf{G}$-module of highest weight $\lambda$.  

Let $W$ be the Weyl group of $\mbf{G}$, and let $\Wext = W \ltimes \bX$ be its extended affine Weyl group.  For $\lambda \in \bX$, let $t_\lambda$ denote the corresponding element of $\Wext$.  For $\lambda \in \bX^+$, let
\[
w_\lambda = \text{the unique element of minimal length in the double coset $W t_\lambda W \subset \Wext$.}
\]
Recall the ``$p$-dilated dot action'' of $\Wext$ on $\bX$: for $w = v \ltimes t_\lambda \in \Wext$ and $\mu \in \bX$, we set
\[
w \cdot \mu = v(\mu + p\lambda + \rho) -\rho,
\]
where, as usual, $\rho$ is one-half the sum of the positive roots.  It is well known (see, for instance,~\cite[Lemma~8.7]{ahr:hcsv}) that
\[
H^\bullet(\mbf{G}_{1},\tilt(\mu)) = 0
\qquad\text{unless}\qquad
\mu = w_\lambda \cdot 0\quad\text{for some $\lambda \in \bX^+$.}
\]
The following theorem describes $H^\bullet(\mbf{G}_{1},\tilt(\mu))$ in the case where $\mu = w_\lambda \cdot 0$ for some $\lambda \in \bX^+$.  It confirms a relative version of a conjecture due to Humphreys~\cite{hum:cmr} (cf.~\cite[Conjecture~8.10]{ahr:hcsv}), as well as part of a refinement of this conjecture proposed by the authors and S.~Riche in~\cite[Conjecture~5.7]{ahr:ctmap}.

\begin{thm}[Relative Humphreys Conjecture]\label{thm:relhc}
For $\lambda \in -\bX^+$, we have
\begin{equation}\label{eqn:relhum-formula}
H^k(\mbf{G}_{1},\tilt(w_\lambda \cdot 0)) \cong
\bigoplus_{i+j = k} R^i\Gamma(\cN,\cS(C,\cT))_j
\end{equation}
where $(C,\cT)$ corresponds to $w_0\lambda$ under the Lusztig--Vogan bijection.  In particular, as a coherent sheaf on $\cN$, $H^\bullet(\mbf{G}_{1}, \tilt(w_\lambda \cdot 0)$ is supported on $\overline{C}$, and 
\[
H^\bullet(\mbf{G}_{1},\tilt(w_\lambda \cdot 0))|_C \cong \cT.
\]
\end{thm}

In this statement, the notation $R^i\Gamma(\cN,\cS(C,\cT))_j$ denotes the $j$th graded component of the grading coming from the $\Gm$-action.

\begin{rmk}
The statement in~\cite[Conjecture~8.10]{ahr:hcsv} describes the support of the cohomology groups $H^\bullet(\mbf{G}_{1}, \tilt(w_\lambda \cdot 0))$ in terms of the two-sided Kazhdan--Lusztig cell containing $w_\lambda$, not in terms of the Lusztig--Vogan bijection.  However, it follows from~\cite[Remark~6]{bez:psaf} combined with the main result of~\cite{ahr:ies} that the nilpotent orbit appearing in~\cite[Conjecture~8.10]{ahr:hcsv} is the same as the one in Theorem~\ref{thm:relhc}.
\end{rmk}

\begin{proof}
An equation similar to~\eqref{eqn:relhum-formula}, but with $\cS(C,\cT)$ replaced by $\fS_\lambda$, is an immediate consequence of~\cite[Proposition~9.1]{ahr:hcsv} (see also~\cite[Lemma~6.3]{ah:cot1} and~\cite[Proposition~9.4]{ar:dkf}).  We obtain~\eqref{eqn:relhum-formula} by combining these results with Theorem~\ref{thm:co-t-struc}.

Regarding $H^\bullet(\mbf{G}_{1}, \tilt(w_\lambda \cdot 0))$ as a coherent sheaf, we can rewrite~\eqref{eqn:relhum-formula} as 
\[
H^\bullet(\mbf{G}_{1},\tilt(w_\lambda \cdot 0)) \cong
\bigoplus_{i} R^i\Gamma(\cN,\cS(C,\cT))\la -i\ra.
\]
The description of the support of $H^\bullet(\mbf{G}_{1},\tilt(w_\lambda \cdot 0))$ and its restriction to $C$ then follow from the properties of $\cS(C,\cT)$ listed in Lemma~\ref{lem:silt-main}.
\end{proof}

%%%%%%%%%%%%%%%%%%%%%%%%%%%%%%%%%%%%%%%%%%%%%%%%%%%%%%%%%%%%%%%%%%%%%%%%%%%
\section{Applications to the \texorpdfstring{$p$}{p}-canonical basis}
\label{sec:pcanonical}
%%%%%%%%%%%%%%%%%%%%%%%%%%%%%%%%%%%%%%%%%%%%%%%%%%%%%%%%%%%%%%%%%%%%%%%%%%%

In this section, we return to allowing $p$ to be pretty good for $G$.  Let $\sC$ denote the Grothendieck group of $\Db\CohGGm(\cN)$.  This is a module over the ring of Laurent polynomials $\Z[v,v^{-1}]$ under the rule $[\cF\la 1\ra] = -v^{-1}[\cF]$. In~\cite{ostrik}, Ostrik defined a certain $\Z[v,v^{-1}]$-basis for $\sC$, denoted by
\[
\{ \uC_\lambda \mid \lambda \in \bX^+ \},
\]
and called the \emph{canonical basis}.  This basis is closely related to the Kazhdan--Lusztig basis of the extended affine Hecke algebra $\cH_\ext$.  Specifically, according to the proof of~\cite[Lemma~2.6]{ostrik}, there is a $\Z[v,v^{-1}]$-linear map $\st: \cH_\ext \to \sC$ whose behavior on the Kazhdan--Lusztig basis $\{\uH_w  \mid w \in \Wext \}$ is given by
\begin{equation}\label{eqn:kl-st}
\st(\uH_w) =
\begin{cases}
\uC_\lambda & \text{if $w$ has minimal length in the coset $Wt_\lambda W \subset \Wext$,} \\
0 & \text{otherwise.}
\end{cases}
\end{equation}
(Note that $\sC$ is \emph{not} an $\cH_\ext$-module.  It is, however, a module over the center $Z(\cH_\ext)$ of $\cH_\ext$, and the map $\st$ is a $Z(\cH_\ext)$-module homomorphism.)  Ostrik conjectured, and Bezrukavikov later proved~\cite{bez:ctm, bez:psaf}, that
\begin{equation}\label{eqn:kl-ic}
\uC_\lambda = [\IC(C,\cL_\omega)]
\end{equation}
where $(C,\omega)$ corresponds to $\lambda$ under the Lusztig--Vogan bijection.

In recent years, the \emph{$p$-canonical basis} for $\cH_\ext$~\cite{jw}, denoted by $\{{}^p\uH_w \mid w \in \Wext\}$, has come to prominence: see, for instance~\cite{amrw,ar:dkf,rw}. It is natural to ask whether $\sC$ has some basis that should be called ``$p$-canonical.''  One approach is to generalize~\eqref{eqn:kl-st} by defining
\begin{equation}\label{eqn:pcan-st}
{}^p \uC_\lambda := \st({}^p\uH_w)
\qquad\text{where $w$ is the element of minimal length in $W t_\lambda W$.}
\end{equation}
(It can be deduced from~\cite{ahr:hcsv} that $\st({}^p\uH_w) = 0$ if $w$ is not minimal in some double coset for $W$.)  On the other hand, in view of Corollary~\ref{cor:silt-char0},~\eqref{eqn:kl-ic} suggests defining
\begin{equation}\label{eqn:pcan-cs}
{}^p \uC_\lambda := [\cS(C,\cT_\omega)]
\end{equation}
via the Lusztig--Vogan bijection.  In fact, Theorem~\ref{thm:co-t-struc} (combined with~\cite{ah:cot1, ahr:hcsv}) implies that these two definitions coincide.

In general, the $p$-canonical basis for $\cH_\ext$ satisfies
\[
{}^p \uH_w \in H_w + \sum_{v < w} \Z_{\ge 0}[v,v^{-1}] H_v,
\]
where ``$<$'' denotes the Bruhat order on $\Wext$.  From~\eqref{eqn:pcan-st} we deduce that
\[
{}^p \uC_\lambda \in \uC_\lambda + \sum_{\mu < \lambda} \Z_{\ge 0}[v,v^{-1}] \uC_\mu,
\]
where ``$<$'' now denotes the usual partial order on weights.  But~\eqref{eqn:pcan-cs} imposes a much strong constraint: if $\lambda$ corresponds to $(C,\omega)$ via $\Theta_\LV$, then
\begin{equation}\label{eqn:pcan-pos}
{}^p \uC_\lambda \in \uC_\lambda + \Bigg(\sum_{\substack{\mu < \lambda \\ \Theta_\LV(\mu) \in \Omega_C}} \Z_{\ge 0} \uC_\mu\Bigg)+ \Bigg(\sum_{C' \subset \partial C} \sum_{\substack{\mu < \lambda\\ \Theta_\LV(\mu) \in \Omega_{C'}}} \Z_{\ge 0}[v,v^{-1}]\uC_\mu\Bigg).
\end{equation}
Note that the first sum has \emph{integer} coefficients, rather than Laurent polynomials.

The extended affine Hecke algebra $\cH_\ext$ is categorified by the monoidal category of Iwahori-equivariant parity sheaves on the dual affine flag variety, denoted by $\Parity_I(\Fl)$.  For $w \in \Wext$, let $\cE_w \in \Parity_I(\Fl)$ denote the corresponding indecomposable parity sheaf.  Given a $G$-stable closed subset $Z \subset \cN$, let
\[
\Parity_I(\Fl)_Z =
\begin{array}{c}
\text{the full additive subcategory generated by $\cE_w[n]$ where $n \in \Z$} \\
\text{and $w$ is either not minimal in any double coset $Wt_\lambda W$,} \\ \text{or $w$ is minimal in some $W t_\lambda W$ such that $\supp \fS_\lambda \subset Z$.}
\end{array}
\]
Roughly,~\eqref{eqn:pcan-pos} says that if $w \in \Wext$ is minimal in the double coset $W t_\lambda W$, then the parity sheaf $\cE_w$ should be ``perverse modulo $\Parity_I(\Fl)_{\partial C}$.''  Given a nilpotent orbit $C \subset \cN$, we now define
\[
\cA_C := \Parity_I(\Fl)_{\overline{C}} / \Parity_I(\Fl)_{\partial C}.
\]
This category inherits from $\Parity_I(\Fl)$ the shift operation $[1]$.  Let
\[
\cA_C^\circ :=
\begin{array}{c}
\text{the full subcategory of $\cA_C$ generated by the} \\
\text{images of $\cE_w \in \Parity_I(\Fl)_{\overline{C}}$ (without shifts).}
\end{array}
\]

We expect that $\cA_C^\circ$ can be made into a monoidal category using Lusztig's ``truncated convolution'' operation $\odot$ (see~\cite{lus:cawgtc}), defined as follows: for $\cF, \cG \in \cA_C^\circ$, let
\[
\cF \odot \cG = {}^p\cH^{\frac{1}{2}\codim C}(\tilde\cF \star^I \tilde\cG) \mod \Parity_I(\Fl)_{\partial C},
\]
where $\tilde\cF$ and $\tilde\cG$ are lifts of $\cF$ and $\cG$ to $\Parity_I(\Fl)_{\overline{C}}$.  Moreover, we conjecture that there is an equivalence of monoidal categories
\[
(\cA_C^\circ, \odot) \cong (\Tilt(G^{x_C}_\red),\otimes).
\]

\appendix
%%%%%%%%%%%%%%%%%%%%%%%%%%%%%%%%%%%%%%%%%%%%%%%%%%%%%%%%%%%%%%%%%%%%%%%%%%%
\section{Co-\texorpdfstring{$t$}{t}-structures preserved by a duality}
\label{sec:duality}
%%%%%%%%%%%%%%%%%%%%%%%%%%%%%%%%%%%%%%%%%%%%%%%%%%%%%%%%%%%%%%%%%%%%%%%%%%%

The following uniqueness result is used in the proof of Theorem~\ref{thm:co-t-struc}.  In this statement, the term ``Tate twist'' is used as in~\cite[\S2.2]{ah:cot1}.

\begin{lem}\label{lem:duality-unique}
Let $\fD$ be a triangulated category equipped with a Tate twist $\lb 1\rb: \fD \to \fD$ and an antiautoequivalence $\delta: \fD^{\mathrm{op}} \simto \fD$ such that $\delta \circ \lb 1\rb \cong \lb {-1}\rb \circ \delta$.  There is at most one silting subcategory $\fS \subset \fD$ with the following properties:
\begin{enumerate}
\item $\fS$ is stable under $\lb 1\rb$ and $\delta$.
\item $\fS$ is a Krull--Schmidt category.
\item For each indecomposable object $S \in \fS$, there is a unique integer $n$ such that $\delta(S\lb n\rb) \cong S\lb n\rb$.
\end{enumerate}
\end{lem}
\begin{proof}
Suppose we have two silting subcategories $\fS, \fS' \subset \fD$ satisfying these properties.  Let $(\fD_{\ge 0}, \fD_{\le 0})$ be the bounded co-$t$-structure on $\fD$ corresponding to $\fS$.

Suppose there exists an object $S' \in \fS'$ such that $S' \notin \fD_{\ge 0}$.  We will derive a contradiction from this.  Let $n$ be the smallest integer such that $S' \in \fD_{\ge n}$, so that $S'[n] \in \fD_{\ge 0}$. (We necessarily have $n < 0$.)  Then it is possible to find a distinguished triangle
\begin{equation}\label{eqn:silt-mindt}
A \xrightarrow{f} S'[n] \xrightarrow{g} B \xrightarrow{h}
\qquad\text{with $A \in \fD_{\ge 1}$ and $B \in \fS$.}
\end{equation}
Choose such a distinguished triangle in which the number of indecomposable direct summands of $B$ is minimized.  This number is well defined because $\fS$ is Krull--Schmidt, and it is nonzero because $S'[n] \notin \fD_{\ge 1}$.

Next, choose an indecomposable summand $T$ of $B$. Let $i: T \to B$ and $p: B \to T$ be the inclusion and projection maps, so that $pi = \id_T$, and $ip$ is idempotent.  

By applying a suitable Tate twist, we may assume without loss of generality that $\delta(T) \cong T$.  Fix an identification of these two objects.  Then $T$ is a direct summand of $\delta(B)$, and $\delta(p): T \to \delta(B)$ and $\delta(i): \delta(B) \to T$ are the inclusion and projection maps, respectively.

Consider the diagram below, in which the rows are distinguished triangles.  (The dotted arrows will be explained later.)
\[
\begin{tikzcd}[row sep=small]
& T \ar[d, "i"] \\
S'[n] \ar[r, "g"] & B \ar[d, "p"] \ar[r, "h"] \ar[ddl, dashed, "\phi'"'] & {A[1]} \ar[r, "{-f[1]}"] \ar[ddl, dashed, "\phi"'] \ar[dd, dashed, "\psi"] & {} \\
& T \ar[d, "\delta(p)"'] \\
\delta(A)[-1] \ar[r,"\delta(h)"] & \delta(B) \ar[d, "\delta(i)"] \ar[r, "\delta(g)"] & \delta(S')[-n] \ar[r, "\delta(f)"] & {} \\
& T
\end{tikzcd}
\]
Because $\fS'$ is a silting subcategory, we have $\Hom(S'[n], \delta(S')[-n]) = 0$.  In particular, $\delta(g)\delta(p)pg = 0$, so $\delta(g)\delta(p)p$ factors through $h$: there exists a map $\psi: A[1] \to \delta(S')[-n]$ such that
\[
\psi h = \delta(g)\delta(p)p.
\]
Because $\fS$ is a silting subcategory, we have $\Hom(A[1], \delta(A)) = 0$, so $\delta(f)\psi = 0$.  Therefore, there is a map $\phi: A[1] \to \delta(B)$ such that
\[
\delta(g)\phi = \psi.
\]
We now have $\delta(g)\phi h = \delta(g)\delta(p)p$, or $\delta(g)(\phi h - \delta(p)p) = 0$.  Therefore, there exists a map $\phi': B \to \delta(A)[-1]$ such that $\delta(h)\phi' = \phi h - \delta(p)p$, or
\begin{equation}\label{eqn:silt-htpy}
\delta(p)p = \phi h - \delta(h)\phi'.
\end{equation}

By the nine lemma, there exist objects $B'$ and $A'$ such that the rows and columns of the following diagram are distinguished triangles:
\[
\begin{tikzcd}
S'[n] \ar[r] \ar[d] & B' \ar[r] \ar[d] & A'[1] \ar[r]\ar[d] & {} \\
S'[n] \ar[r,"g"] \ar[d] & B \ar[r, "h"] \ar[d, "\delta(i)\phi h"'] & A[1] \ar[r, "{-f[1]}"] \ar[d, "\delta(i)\phi" ] & {} \\
0 \ar[r] \ar[d] & T \ar[r, equal] \ar[d] & T \ar[r] \ar[d] & {} \\
{} & {} & {}
\end{tikzcd}
\]
The rightmost column shows that $A' \in \fD_{\ge 1}$.  Now consider the map $\delta(i)\phi h i \in \End(T)$.  If this were an isomorphism, then the middle column above would split (by the map $i(\delta(i)\phi h i)^{-1}: T \to B$), and $B'$ would be a direct summand of $B$, and hence an object of $\fS$ with fewer indecomposable summands than $B$.  But then the top row would contradict the minimality in the choice of~\eqref{eqn:silt-mindt}.  We conclude that $\delta(i)\phi h i$ is not invertible.

Similar reasoning shows that $-\delta(i)\delta(h)\phi' i \in \End(T)$ is also not invertible.  Since $\End(T)$ is a local ring, the sum of two nonunits is again a nonunit, and thus $\delta(i)\phi h i - \delta(i) \delta(h) \phi' i$ is not invertible.  But by~\eqref{eqn:silt-htpy}, we have
\[
\delta(i)\phi h i - \delta(i) \delta(h) \phi' i = \delta(i)\delta(p)p i = \id_T,
\]
a contradiction.  

We conclude that $S' \in \fD_{\ge 0}$, and hence that $\fS' \subset \fD_{\ge 0}$.  Applying $\delta$, we also have $\fS' \subset \fD_{\le 0}$, so $\fS' \subset \fS$.  Reversing the roles of the two categories, we also obtain $\fS' \supset \fS$, and hence $\fS = \fS'$.
\end{proof}

\end{document}